\RequirePackage{fix-cm}
\documentclass[smallextended]{svjour3}       
\smartqed  
\usepackage{graphicx}
\smartqed  
\usepackage{amsfonts}
\usepackage{amssymb}
\usepackage{epstopdf}
\usepackage{bm}
\usepackage{mathtools}
\usepackage{lipsum}
\usepackage{multirow}
\usepackage{subfigure}
\ifpdf
\DeclareGraphicsExtensions{.eps,.pdf,.png,.jpg}
\else
\DeclareGraphicsExtensions{.eps}
\fi

\newtheorem{assumption}{Assumption}
\newcommand{\triplenorm}[1]{%
	\left\vert\kern-0.9pt\left\vert\kern-0.9pt\left\vert #1
	\right\vert\kern-0.9pt\right\vert\kern-0.9pt\right\vert}
\numberwithin{equation}{section}
\usepackage{chngcntr} 
\numberwithin{theorem}{section} 
\numberwithin{lemma}{section}
\numberwithin{proposition}{section}
\numberwithin{remark}{section}
\numberwithin{assumption}{section}
\begin{document}
	
	\title{A Hybridizable  Discontinuous Galerkin Method for Magnetic Advection-Diffusion Problems\thanks{
	The work of Shuonan Wu is supported in part by the Beijing Natural Science Foundation No. 1232007 and the National Natural Science Foundation of China grant No. 12222101.}}
	
	\titlerunning{HDG for $H({\rm curl})$ Advection-Diffusion Problems}        
	
	\author{Jindong Wang \and
		Shuonan Wu 
	}
	
	
	\institute{Jindong Wang \at
		School of Mathematical Sciences, Peking University, Beijing 100871, China \\
		\email{jdwang@pku.edu.cn}           
		\and
		Shuonan Wu \at
		School of Mathematical Sciences, Peking University, Beijing 100871, China \\
		\email{snwu@math.pku.edu.cn}
	}
	
	\date{Received: date / Accepted: date}

	\maketitle
	
	\begin{abstract}
	We propose and analyze a hybridizable discontinuous Galerkin (HDG) method for solving a mixed magnetic advection-diffusion problem within a more general Friedrichs system framework. With carefully constructed numerical traces, we introduce two distinct stabilization parameters: $\tau_t$ for the tangential trace and $\tau_n$ for the normal trace. These parameters are tailored to satisfy different requirements, ensuring the stability and convergence of the method. Furthermore, we incorporate a weight function to facilitate the establishment of stability conditions. We also investigate an elementwise postprocessing technique that proves to be effective for both two-dimensional and three-dimensional problems in terms of broken $\bm{H}({\rm curl})$ semi-norm accuracy improvement. Extensive numerical examples are presented to showcase the performance and effectiveness of the HDG method and the postprocessing techniques.
	\subclass{65N30 \and  65N12}
	\end{abstract}
	
\section{Introduction}\label{sec:introduction}
In this paper, we consider the following magnetic advection-diffusion equation:
\begin{equation}\label{eq:original}
	\begin{aligned}
		\nabla\times(\varepsilon\nabla\times\bm{u})-\bm{\beta}\times(\nabla\times\bm{u})+\nabla(\bm{\beta}\cdot\bm{u})+\gamma \bm{u}&=\bm{f} \quad & \text{ in }\Omega,\\
		\bm{n}\times\bm{u}+ \chi_{\Gamma^-} (\bm{u} \cdot \bm{n}) \bm{n}&=\bm{g} \quad & \text{ on }\Gamma,
	\end{aligned}
\end{equation}
where $\Omega\in\mathbb{R}^3$ is a bounded polyhedral domain with boundary $\Gamma=\partial\Omega$. The inflow and outflow parts of $\Gamma$ are defined as follow:
$$
\Gamma^-:=\{\bm{x}\in\Gamma:\bm{\beta}(\bm{x})\cdot\bm{n}(\bm{x})<0\},\quad \Gamma^+:=\{\bm{x}\in\Gamma:\bm{\beta}(\bm{x})\cdot\bm{n}(\bm{x})\ge0\}.
$$ In this context, $\varepsilon$ is a positive constant, $\bm{f}\in
\bm{L}^2(\Omega)$, $\bm{g}\in \bm{H}^{3/2}(\Gamma)$ and $\bm{n}$ denotes the unit outward normal to $\Gamma$.

The magnetic advection-diffusion problem (\ref{eq:original}) is of great importance in many areas of modern sciences and engineering, especially the magnetohydrodynamics (MHD) \cite{gerbeau2006mathematical}, which involves understanding the behavior of magnetic fields in the presence of fluid flow (advection) and magnetic diffusion. 

It is well-known that the scalar advection-diffusion problem exhibits a boundary layer phenomenon in the case of advection dominance, specifically, $\varepsilon\ll |\bm{\beta}|$. This behavior is also evident in the model problem (\ref{eq:original}) and introduces numerical challenges in obtaining a stable and accurate solution. A considerable body of literature has addressed these challenges, proposing various methods for the scalar problem. One class of methods involves the concept of exponential fitting, which utilizes exponential spline functions to capture the presence of the boundary layer. This approach encompasses the construction and analysis of spline basis on both rectangular grids \cite{o1991globally,roos1996novel,dorfler1999uniform,dorfler1999uniform2} and unstructured grids \cite{sacco1998finite,sacco1999nonconforming,wang1997novel}. Additionally, exponential functions can be incorporated in the assembly of the stiffness matrix based on operator fitting ideas \cite{xu1999monotone}. Another class of methods follows the principle of upwind or streamline stabilization. Well-known residual-based stabilization techniques include streamline upwind Petrov Galerkin (SUPG) method \cite{hughes1979multidimentional,brooks1982streamline,niijima1989pointwise,johnson1987crosswind}, Galerkin least squares finite element method \cite{hughes1989new,franca1992stabilized} and bubble function stabilization \cite{baiocchi1993virtual,brezzi1994choosing,brezzi1998applications,brezzi1998further,franca2002stability}. Additionally, symmetric stabilization techniques have been developed, such as Local Projection Stabilization (LPS) \cite{becker2001finite,braack2006local,matthies2007unified,ganesan2010stabilization} and continuous interior penalty (CIP) \cite{burman2004edge,burman2005unified,burman2006continuous}.

Considerable attention has been given to the discontinuous Galerkin (DG) methods for scalar advection-diffusion problems over the past decades. These methods have been extensively studied and applied in the literature, as evidenced by works such as  \cite{ayuso2009discontinuous,cockburn1999discontinuous,cockburn1998local,houston2002discontinuous,zarin2005interior}. It is worth noting that conventional DG methods have a notable drawback compared to conforming methods: they require a higher number of globally-coupled degrees of freedom for the same mesh. However, a special class of the discontinuous Galerkin method known as the hybridizable discontinuous Galerkin (HDG) method \cite{cockburn2009unified} for elliptic problems has also emerged as a notable approach in this field and offers a distinct advantage over traditional DG methods by limiting the globally-coupled degrees of freedom to the numerical traces on the mesh skeleton. This approach has attracted significant research attention, and HDG methods for scalar advection-diffusion problems have been widely investigated in works such as \cite{chen2012analysis,chen2014analysis,cockburn2009hybridizable,fu2015analysis,qiu2016hdg,nguyen2009implicit}. These studies delve into the analysis and development of HDG methods, exploring their potential for addressing the challenges posed by scalar advection-diffusion problems.

Recently, the application of the convection-diffusion model in vector field problems, such as electromagnetic fields, has become increasingly important. The complex mathematical form and structure of the convection term in vector fields necessitate further research in this area. In the context of exterior calculus, the magnetic advection-diffusion problem and scalar advection-diffusion problem are essentially specific situations of an advection-diffusion problem for different types of differential forms, which may share common characteristics. Motivated by \cite{xu1999monotone}, a simplex-averaged finite element (SAFE) method was proposed for both scalar and vector cases of advection-diffusion problems in \cite{wu2020simplex}. This method belongs to the class of exponential fitting methods, which naturally incorporate an upwind effect without the need for any stabilization parameters. Heumann and Hiptmair \cite{heumann2011eulerian,heumann2013stabilized,heumann2016stabilized} successfully applied a special class of upwind methods in vector fields. A class of the primal discontinuous Galerkin methods for magnetic advection-diffusion problems based on the weighted-residual approach is given in \cite{wang2022discontinuous}. However, there is currently no literature available on the use of the HDG method for magnetic advection-diffusion problems, despite extensive research on the HDG methods for Maxwell equations \cite{chen2017superconvergent,chen2018priori,chen2019analysis,du2020unified} and time-harmonic Maxwell equations \cite{nguyen2011hybridizable,lu2017absolutely}. Among them, a unified error analysis of HDG methods for the static Maxwell equations is provided in \cite{du2020unified}.

In this paper, we present a HDG method for the magnetic advection-diffusion problem. To accomplish this, we introduce an auxiliary variable $\bm{w}=\varepsilon\nabla\times\bm{u}$ as a new unknown. Our work presents, for the first time, a mixed formulation for expressing the magnetic advection-diffusion problem (\ref{eq:original}), as outlined below:
\begin{equation}\label{eq:mixed}
	\begin{aligned}
		\varepsilon^{-1}\bm{w}-\nabla\times\bm{u}&=0 \text{ in }\Omega,\\
		\nabla\times\bm{w}-\bm{\beta}\times(\nabla\times\bm{u})+\nabla(\bm{\beta\cdot\bm{u}})+\gamma\bm{u}&=\bm{f}\text{ in }\Omega,\\
		\bm{n}\times\bm{u}+\chi_{\Gamma^-}(\bm{u}\cdot\bm{n})\bm{n}&=\bm{g} \text{ on }\partial \Omega.
	\end{aligned}
\end{equation}

As a key component of HDG method, we carefully construct numerical traces, which play a crucial role in the stability. Unlike the scalar case, in the vector field scenario, we  introduce two distinct stabilization parameters: $\tau_t$ for the tangential trace and $\tau_n$ for the normal trace. The presence of these two parameters is important for ensuring stability and convergence of the method. It is worth emphasizing that $\tau_t$ and $\tau_n$ serve different purposes in achieving stability and convergence, thus the requirements for $\tau_t$ and $\tau_n$ differ accordingly. This difference can  provide some insights for understanding the underlying mechanism of stabilization.

The HDG method provides an energy identity that demonstrates its stability at first glance. However, the scheme analysis is established under a general assumption regarding variable advection and reaction. This assumption allows for a degenerate Friedrichs system, where the minimal eigenvalue of the Friedrichs system only needs to be nonnegative (see Assumption \ref{ass:reaction}). As indicated in the energy identity, the effective reaction matrix is degenerate and does not provide effective control over the $L^2$ norm. Therefore, drawing inspiration from \cite{ayuso2009discontinuous} and \cite{fu2015analysis}, we introduce a weight function to obtain a suitable test function that ensures inf-sup stability. Subsequently, we derive a priori error analysis for the proposed HDG method in terms of the energy norm.

As widely known, HDG methods for scalar diffusion-dominated problems exhibit superconvergence properties. These properties allow for local postprocessing, resulting in a new approximation that possibly converges with a better accuracy. Such techniques have been extensively utilized in various studies, including \cite{cockburn2009superconvergent,cockburn2010projection,nguyen2011hdg}. For the Maxwell problem, a postprocessing method is proposed in \cite{nguyen2011hybridizable}, aiming at an improved accuracy in the $\bm{H}({\rm curl})$-norm. However, this postprocessing method is limited to two-dimensional problems and cannot be extended to three-dimensional vector problems. Hence, in this paper, we present a local postprocessing method that is suitable for both two-dimensional and three-dimensional problems. In the diffusion-dominated case, this method demonstrate similar accuracy improvements to the findings in \cite{nguyen2011hybridizable}, as evidenced by numerical tests.

The remaining sections of this paper are organized as follows. Section \ref{sec:preliminary} provides the necessary preliminary results, including the assumptions and notation used throughout the paper. In Section \ref{sec:HDG}, we present the HDG method for solving the mixed magnetic advection-diffusion problem. We will discuss the requirements for the stabilization parameters, present the characterization of the Hybridizable Discontinuous Galerkin (HDG) method, and propose an energy norm that will be utilized for further analysis. Section \ref{sec:convergence} focuses on the stability analysis and provides a priori error estimates based on the energy norm. In Section \ref{sec:postprocessing}, we introduce various postprocessing methods aimed at achieving improved broken $\bm{H}({\rm curl})$ approximations. Finally, in Section \ref{sec:numerical}, we present numerical results to evaluate the performance of the HDG method and the postprocessing techniques.

\section{Preliminaries} \label{sec:preliminary}
In this section, we present several assumptions that are crucial for establishing the well-posedness of our method. Additionally, we introduce some notation to enhance a clear understanding of the method. Throughout the paper, we adhere to the standard notation and definitions for Sobolev spaces (cf. \cite{adams2003sobolev}).

\subsection{Assumptions on velocity field}
As is proposed by \cite{ayuso2009discontinuous}, we adopt an assumption on the velocity field $\bm{\beta}$.

\begin{assumption}[neither closed curve nor stationary point]\label{ass:beta}
	The velocity field $\bm{\beta}\in W^{1,\infty}(\Omega)$ has neither closed curves nor stationary points, i.e., 
	\begin{equation}
		\bm{\beta}\in W^{1,\infty}(\Omega) \text{ has no closed curves,} \quad \bm{\beta}(\bm{x})\neq \bm{0} \quad \forall \bm{x}\in \Omega.
	\end{equation}
\end{assumption} 
This implies that there exists a smooth function  $\psi$ depending on $\bm{\beta}$ so that
$$
\bm{\beta}\cdot\nabla\psi(\bm{x})\ge b_0, \quad \bm{x}\in \Omega,
$$
for some constant $b_0>0$. See \cite{devinatz1974asymptotic} or \cite[Appendix A]{ayuso2009discontinuous} for a proof.

The subsequent assumption gives special emphasis on the magnetic advection operator, namely the Lie derivative,
\begin{equation} \label{eq:Lie-advection}
	L_{\bm{\beta}} \bm{u} := - \bm{\beta} \times (\nabla \times \bm{u}) +
	\nabla (\bm{\beta} \cdot \bm{u}),
\end{equation}
and its formal dual operator 
\begin{equation} \label{eq:dual-Lie}
	\mathcal{L}_{{\bm{\beta}}} \bm{u} := \nabla \times ({\bm{\beta}}
	\times \bm{u}) - {\bm{\beta}} \nabla \cdot \bm{u},
\end{equation}
which yields, after a direct calculation, that for any domain $D$
\begin{subequations} \label{eq:Lie}
	\begin{align}
		L_{{\bm{\beta}}} \bm{u}+ \mathcal{L}_{{\bm{\beta}}}\bm{u}  &= -
		(\nabla \cdot {\bm{\beta}}) \bm{u} + \left[ \nabla {\bm{\beta}} +
		(\nabla {\bm{\beta}})^T \right] \bm{u}, \label{eq:Lie-1} \\
		(L_{{\bm{\beta}}} \bm{u}, \bm{v})_D &= (\bm{u},
		\mathcal{L}_{{\bm{\beta}}}\bm{v})_D  + \langle \bm{\beta} \cdot
		\bm{n}, \bm{u} \cdot \bm{v} \rangle_{\partial D}, \quad \forall
		\bm{u}, \bm{v} \in \bm{H}^{1}(D). \label{eq:Lie-2}
	\end{align}
\end{subequations}

Similar to \cite{wang2022discontinuous}, we present the following assumption regarding the minimum eigenvalue of a degenerate ``effective'' reaction matrix $\mathcal{R}$.
\begin{assumption}[degenerate Friedrichs system]\label{ass:reaction}
	We assume that $\gamma\in L^\infty(\Omega)$ and $\bm{\beta}\in \bm{W}^{1,\infty}(\Omega)$ are such that
\begin{equation} \label{eq:Friedrichs}
	\rho(\bm{x}) :=  \lambda_{\min}(\mathcal{R}) :=\lambda_{\min} \left[ 
	(\gamma - \frac{\nabla \cdot \bm{\beta}}{2}) I + \frac{\nabla
		\bm{\beta} + (\nabla \bm{\beta})^T}{2} \right] \geq \rho_0 \geq
	0, \quad \bm{x}\in \Omega.
\end{equation}
\end{assumption}

\subsection{Mesh and approximation spaces}
Here, we introduce the notation that will be used to describe the HDG method. Let $\mathcal{T}_h$ be a conforming triangulation of $\Omega$ made of shape-regular and quasi-uniform simplicial elements; The diameter of each element $T$ is denoted as $h_T$, and $h=\max_{T\in\mathcal{T}_h}h_T$. Given an element $T\in\mathcal{T}_h$, $\partial T$ represents the set of its facets $F$. The set of all (boundary) facets is denoted by $\mathcal{F}_h$ ($\mathcal{F}_h^\partial$), and $\mathcal{F}_h^{\partial,-}$ and $\mathcal{F}_h^{\partial,+}$ are subsets of $\mathcal{F}_h^\partial$, representing the facets on the inflow and outflow boundaries, respectively.  Unlike $\mathcal{F}_h$, $\partial\mathcal{T}_h$ is used to represent the union of $\partial T$. For functions $\bm{u}$ and $\bm{v}$, we write
$$
(\bm{u},\bm{v})_{\mathcal{T}_h}:=\sum_{T\in\mathcal{T}_h}(\bm{u},\bm{v})_T,\quad\langle \bm{u},\bm{v}\rangle_{\partial\mathcal{T}_h}:=\sum_{T\in\mathcal{T}_h}\langle \bm{u},\bm{v}\rangle_{\partial T},
$$
where $(\cdot,\cdot)_D$ denotes the integral over the domain $D\subset\mathbb{R}^d$, and $\langle\cdot,\cdot\rangle_D$ denotes the integral over $D\subset\mathbb{R}^{d-1}$. 

Let $\bm{\mathcal{P}}_k(D)$ denote the set of polynomials of degree at most $k ~(k\ge 0)$ on a domain $D$. We define the following finite element spaces:
\begin{subequations}
\begin{align}
	\bm{W}_h&:=\{\bm{w}\in \bm{L}^2(\Omega):\bm{w}|_{T}\in \bm{\mathcal{P}}_k(T),\forall T\in\mathcal{T}_h\},\\
	\bm{V}_h&:=\{\bm{v}\in \bm{L}^2(\Omega):\bm{v}|_{T}\in \bm{\mathcal{P}}_k(T),\forall T\in\mathcal{T}_h\},\\
	\bm{M}_h&:=\{\bm{\mu}\in \bm{L}^2(\mathcal{F}_h):\bm{\mu}|_{F}\in \bm{\mathcal{P}}_k(F),\forall F\in\mathcal{F}_h\}.
\end{align}
\end{subequations}
Considering the boundary condition, we also define $\bm{M}_h(\bm{g}) :=\{\bm{\mu}\in \bm{M}_h:\langle{\bm{n}\times\bm\mu}+\chi_{\Gamma^-}(\bm{\mu}\cdot\bm{n})\bm{n},\bm{\xi}\rangle_{\mathcal{F}_h^\partial }=\langle\bm{g},\bm{\xi}\rangle_{\mathcal{F}_h^\partial },\forall\bm{\xi}\in\bm{M}_h\}$.

\section{The HDG method} \label{sec:HDG}
\subsection{Formulation}
In light of \eqref{eq:mixed}, the HDG method seeks an approximation $(\bm{w}_h,\bm{u}_h,\widehat{\bm{u}}_h)\in \bm{W}_h\times\bm{V}_h\times\bm{M}_h$ such that
\begin{subequations}\label{eq:hdg}
	\begin{align}
	(\varepsilon^{-1}\bm{w}_h,\bm{r})_{\mathcal{T}_h}-(\bm{u}_h,\nabla\times\bm{r})_{\mathcal{T}_h}+\langle \widehat{\bm{u}}_h,\bm{n}\times\bm{r}\rangle_{\partial\mathcal{T}_h}&=0,\label{eq:hdg1}\\
	\begin{aligned}
		(\bm{w}_h,\nabla\times\bm{v})_{\mathcal{T}_h}+\langle\bm{n}\times\widehat{\bm{w}}_h,\bm{v}\rangle_{\partial\mathcal{T}_h}+(\bm{u}_h,\mathcal{L}_{\bm{\beta}}\bm{v}+\gamma\bm{v})_{\mathcal{T}_h}\quad  &\\ 
		+ \langle\bm{n}\times\widehat{\bm{u}}_h,\bm{\beta}\times\bm{v}\rangle_{\partial\mathcal{T}_h}+\langle \widehat{\bm{\beta}\cdot\bm{u}_h},\bm{n}\cdot\bm{v}\rangle_{\partial\mathcal{T}_h}	\end{aligned}&=(\bm{f},\bm{v})_{\mathcal{T}_h},\\
		-\langle \bm{n}\times\widehat{\bm{w}}_h-\bm{\beta}\times(\bm{n}\times\widehat{\bm{u}}_h)+\widehat{\bm{\beta}\cdot\bm{u}_h}\bm{n},\bm{\mu}\rangle_{\partial\mathcal{T}_h\backslash\mathcal{F}_h^\partial }&=0,\label{eq:hdgtrans}\\
		\langle\bm{n}\times\widehat{\bm{u}}_h+ \chi_{\Gamma^-} (\widehat{\bm{u}}_h \cdot \bm{n}) \bm{n}-\chi_{\Gamma^+}\tau_n[(\bm{u}_h-\widehat{\bm{u}}_h)\cdot\bm{n}]\bm{n},\bm{\mu}\rangle_{\mathcal{F}_h^\partial }&=\langle \bm{g},\bm{\mu}\rangle_{\mathcal{F}_h^\partial },
	\end{align}
\end{subequations}
for all $(\bm{r},\bm{v},\bm{\mu})\in \bm{W}_h\times\bm{V}_h\times\bm{M}_h$, where the numerical traces defined on $\partial \mathcal{T}_h$ are given by
\begin{subequations}
	\begin{align}
		\widehat{\bm{w}}_h&=\bm{w}_h+\tau_t(\bm{u}_h-\widehat{\bm{u}}_h)\times \bm{n},\\
		\widehat{\bm{\beta}\cdot\bm{u}_h}&=\bm{\beta}\cdot\widehat{\bm{u}}_h+\tau_n(\bm{u}_h-\widehat{\bm{u}}_h)\cdot \bm{n}.
	\end{align}
\end{subequations}
The stabilization parameters $\tau_t$ and $\tau_n$ are piecewise nonnegative constants defined on each $F\in\partial \mathcal{T}_h$. Note that these two stabilization parameters are double-valued on facet $F$, depending on the perspective from which the adjacent element is observed.

\subsection{Requirements for stabilization parameters}
To derive the convergence analysis in the next section, we introduce certain requirements on the stabilization parameters $\tau_t$ and $\tau_n$, that is, on each $F \in \partial \mathcal{T}_h$ 
\begin{subequations}\label{eq:asstau}
	\begin{align}
		0& \le\tau_t,\tau_n\le C_{up},\label{eq:asstau1}\\
		\inf_{\bm{x} \in F} \left( \tau_t-\frac{1}{2}\bm{\beta}\cdot\bm{n} \right) &\ge C_t \max_{\bm{x} \in F} |\bm{\beta}\cdot\bm{n}|,\label{eq:asstau2}\\
		\inf_{\bm{x}\in F} \left( \tau_n-\frac{1}{2}\bm{\beta}\cdot\bm{n} \right) &\ge C_n \max\{\max_{\bm{x} \in F} |\bm{\beta}\cdot\bm{n}|, C_b\},\label{eq:asstau3}\\
		\tau_t &\ge C_w\min\{\frac{\varepsilon}{h},1\},\label{eq:asstau4}
	\end{align}
\end{subequations}
for some positive constants $C_{up},C_t,C_n$, $C_b$ and $C_w$ independent of $\varepsilon$ and $h$.

It is worth noting that the condition \eqref{eq:asstau2} implies that, for every $F \in \partial T$, we have the following relation:
$$ 
\min_{\bm{x}\in F}|\tau_t - \frac12\bm{\beta} \cdot \bm{n}| \ge C \left( 
\min_{\bm{x}\in F}|\tau_t - \frac12\bm{\beta} \cdot \bm{n}|  + \max_{\bm{x}\in F} |\bm{\beta}\cdot \bm{n}|
\right) \ge C\tau_t,
$$ 
whence 
\begin{equation} \label{eq:min-max-equiv}
\max_{\bm{x}\in F} |\tau_t - \frac12\bm{\beta} \cdot \bm{n} | \leq \tau_t + \frac12 \max_{\bm{x}\in F} |\bm{\beta}\cdot\bm{n}| \le C\min_{\bm{x}\in F}|\tau_t - \frac12\bm{\beta} \cdot \bm{n}|.
\end{equation}
Similarly, by examining equation \eqref{eq:asstau3}, we observe that the above condition also holds for the normal component. In other words, if we replace $\tau_t$ with $\tau_n$ in \eqref{eq:min-max-equiv}, the same condition applies.

\subsection{A characterization of the HDG method}
We begin by expressing  the unknowns $\bm{w}_h$ and $\bm{u}_h$ in terms of the unknown $\widehat{\bm{u}}_h$.
Given $\bm{\lambda}\in \bm{L}^2(\mathcal{F}_h)$ and $\bm{f}\in\bm{L}^2(\mathcal{T}_h)$, consider the set of local problems in each $T\in\mathcal{T}_h$: find
	$$
	(\bm{w}_h,\bm{u}_h) \in \bm{W}(T)\times \bm{V}(T),
	$$
	where $\bm{W}(T):=\bm{\mathcal{P}}_k(T)$ and $\bm{V}(T):=\bm{\mathcal{P}}_k(T)$, such that
	$$
	\begin{aligned}
		(\varepsilon^{-1}\bm{w}_h,\bm{r})_{T}-(\bm{u}_h,\nabla\times\bm{r})_{T}&=-\langle \bm{\lambda},\bm{n}\times\bm{r}\rangle_{\partial T},\\
		\begin{aligned}
			(\nabla\times\bm{w}_h,\bm{v})_{T}+(\bm{u}_h,\mathcal{L}_{\bm{\beta}}\bm{v}+\gamma\bm{v})_{T}\quad  &\\ 
			+\langle\tau_t{\bm{u}}_h^t,\bm{v}^t\rangle_{\partial T}+\langle \tau_n \bm{u}_h^n,\bm{v}^n\rangle_{\partial T}	
			\end{aligned}& = 
			\begin{aligned}
			&(\bm{f},\bm{v})_{T} -\langle (\tau_n-\bm{\beta}\cdot\bm{n}){\bm\lambda}^n,\bm{v}^n\rangle_{\partial T}\\
			& ~-\langle (\tau_t-\bm{\beta}\cdot\bm{n}){\bm\lambda}^t,\bm{v}^t\rangle_{\partial T}
		\end{aligned},\\
	\end{aligned}
	$$
	for all $(\bm{r},\bm{v})\in \bm{W}(T)\times \bm{V}(T)$, where $\bm{\lambda}^t:=(\bm{n}\times\bm{\lambda})\times\bm{n}$ and $\bm{\lambda}^n:=(\bm{n}\cdot\bm{\lambda})\bm{n}$.
	
	We denote the solution of the above local problem when $\bm{\lambda}=\bm{0}$ as $(\bm{w}_h^{\bm{f}},\bm{u}_h^{\bm{f}})$. Similarly, we denote the solution when $\bm{f}=\bm{0}$ as $(\bm{w}_h^{\bm{\lambda}},\bm{u}_h^{\bm{\lambda}})$. Therefore, we can express $(\bm{w}_h,\bm{u}_h)$ as follows:
	$$
	(\bm{w}_h,\bm{u}_h) = (\bm{w}_h^{\bm{f}},\bm{u}_h^{\bm{f}}) + (\bm{w}_h^{\bm{\lambda}},\bm{u}_h^{\bm{\lambda}}).
	$$ 
	If we set $\bm{\lambda}:=\widehat{\bm{u}}_h$, we observe that $(\bm{w}_h,\bm{u}_h)$ can be expressed in terms of $\widehat{\bm{u}}_h$ and $\bm{f}$. Consequently, we can eliminate these two known quantities from the equations and solve solely for $\widehat{\bm u}_h$. By applying additional inductions, we can have that $\widehat{\bm{u}}_h\in \bm{M}_h(\bm{g})$ must satisfy
	\begin{equation} \label{eq:HDG-characterization}
	a_h(\widehat{\bm{u}}_h,\bm{\mu})=b_h(\bm{\mu}), \quad \forall \bm{\mu}\in \bm{M}_h(\bm{0}),
	\end{equation}
	where
	$$
	\begin{aligned}
		a_h(\bm{\lambda},\bm{\mu})&:=-
		\langle\bm{n}\times\bm{w}_h^{\bm{\lambda}},\bm{\mu}\rangle_{\partial \mathcal{T}_h}-\langle \tau_t(\bm{u}_h^{\bm{\lambda},t}-\bm{\lambda}^t),\bm{\mu}^t\rangle_{\partial \mathcal{T}_h}-\langle \tau_n(\bm{u}_h^{\bm{\lambda},n}-\bm{\lambda}^n),\bm{\mu}^n\rangle_{\partial \mathcal{T}_h},\\
		b_h(\bm{\mu})&:=\langle \bm{n}\times\bm{w}_h^{\bm{f}},\bm{\mu}\rangle_{\partial\mathcal{T}_h}+\langle\tau_t\bm{u}_h^{\bm{f},t},\bm{\mu}^t\rangle_{\partial \mathcal{T}_h} + \langle\tau_n\bm{u}_h^{\bm{f},n},\bm{\mu}^n\rangle_{\partial \mathcal{T}_h}.
	\end{aligned}
	$$
	
	Indeed, we use the facts that for all $\bm{\lambda} \in \bm{M}_h(\bm{g})$ and $\bm{\mu} \in\bm{M}_h(\bm{0})$, 
	$$
	\langle (\bm{\beta} \cdot \bm{n})\bm{\lambda}, \bm{\mu} \rangle_{\partial \mathcal{T}_h \setminus \mathcal{F}_h^\partial} = 0, \quad 
	\bm{n} \times \bm{\mu}|_{\mathcal{F}_h^\partial} = \bm{0}, \quad  \langle\tau_n(\bm{u}_h^{\bm{f},n} + \bm{u}_h^{\bm{\lambda},n}-\bm{\lambda}^n),\bm{\mu}^n\rangle_{ \mathcal{F}_h^\partial} = 0.
	$$
	The formulation \eqref{eq:HDG-characterization} is equivalent to the transmission condition \eqref{eq:hdgtrans}. As a result, we obtain a characterization of the HDG method in terms of $\widehat{\bm u}_h$.
		 
\subsection{Energy norm}
For the HDG method, the summation of the left-hand side terms of \eqref{eq:hdg} yields the following bilinear form:
\begin{equation}\label{eq:bilinear}
	\begin{aligned}
		&B\big((\bm{w},\bm{u},\bm{\lambda}),(\bm{r},\bm{v},\bm{\mu})\big):=\\&(\varepsilon^{-1}\bm{w},\bm{r})_{\mathcal{T}_h}-(\bm{u},\nabla\times\bm{r})_{\mathcal{T}_h}+\langle\bm{\lambda},\bm{n}\times\bm{r}\rangle_{\partial\mathcal{T}_h}+(\bm{u},\mathcal{L}_{\bm{\beta}}\bm{v}+\gamma\bm{v})_{\mathcal{T}_h}\\
		&~+\langle(\bm{\beta}\cdot\bm{n})\bm{\lambda},\bm{v}-\bm{\mu}\rangle_{\partial\mathcal{T}_h}	+(\bm{w},\nabla\times\bm{v})_{\mathcal{T}_h}+\langle \bm{n}\times\bm{w},\bm{v}-\bm{\mu}\rangle_{\partial\mathcal{T}_h}\\
		&~+\langle\tau_t \bm{n}\times(\bm{u}-\bm{\lambda}),\bm{n}\times(\bm{v}-\bm{\mu})\rangle_{\partial\mathcal{T}_h}+\langle\tau_n \bm{n}\cdot(\bm{u}-\bm{\lambda}),\bm{n}\cdot(\bm{v}-\bm{\mu})\rangle_{\partial\mathcal{T}_h}\\
		&~+\langle(\bm{\beta}\cdot\bm{n})\bm{\lambda}\cdot\bm{n},\bm{\mu}\cdot\bm{n}\rangle_{\mathcal{F}_h^{\partial,+}}.
	\end{aligned}
\end{equation}
Then, the HDG method can be reformulated as follows: Find $(\bm{w}_h,\bm{u}_h,\widehat{\bm{u}}_h)\in\bm{W}_h\times\bm{V}_h\times\bm{M}_h(\bm{g})$ such that
\begin{equation}
	B\big((\bm{w}_h,\bm{u}_h,\widehat{\bm{u}}_h),(\bm{r},\bm{v},\bm{\mu})\big)=(\bm{f},\bm{v})_{\mathcal{T}_h},
\end{equation}
for all $(\bm{r},\bm{v},\bm{\mu})\in \bm{W}_h\times\bm{V}_h\times\bm{M}_h(\bm{0})$.

If $\bm{g}\equiv\bm{0}$, we can induce that $\widehat{\bm{u}}_h\times\bm{n}\equiv\bm{0}$ on $\mathcal{F}_h^\partial$ and $\widehat{\bm{u}}_h\cdot \bm{n}\equiv0$ on $\mathcal{F}_h^{\partial,-}$. Exploiting this property and choosing $\bm{r}=\bm{w}_h$, $\bm{v}=\bm{u}_h$, and $\bm{\mu}=\widehat{\bm{u}}_ h$, we can integrate by parts in \eqref{eq:bilinear} to derive the following energy equality: 
\begin{equation}
\begin{aligned}
		& (\varepsilon^{-1}\bm{w}_h,\bm{w}_h)_{\mathcal{T}_h}+\langle (\tau_t-\frac{1}{2}\bm{\beta}\cdot\bm{n})\bm{n}\times(\bm{u}_h-\widehat{\bm{u}}_h),\bm{n}\times (\bm{u}_h-\widehat{\bm{u}}_h)\rangle_{\partial\mathcal{T}_h}\\
		&+\langle (\tau_n-\frac{1}{2}\bm{\beta}\cdot\bm{n})\bm{n}\cdot(\bm{u}_h-\widehat{\bm{u}}_h),\bm{n}\cdot (\bm{u}_h-\widehat{\bm{u}}_h)\rangle_{\partial\mathcal{T}_h} \\	
		&+\frac{1}{2} \langle (\bm{\beta} \cdot\bm{n})\widehat{\bm{u}}_h \cdot\bm{n},
		\widehat{\bm{u}}_h\cdot \bm{n}\rangle_{\mathcal{F}_h^{\partial,+}}+(\mathcal{R}\bm{u}_h,\bm{u}_h)_{\mathcal{T}_h}=(\bm{f},\bm{u}_h)_{\mathcal{T}_h}.
	\end{aligned}
\end{equation}
Here, we have  $\langle (\bm{\beta} \cdot\bm{n})\widehat{\bm{u}}_h \cdot\bm{n},
\widehat{\bm{u}}_h\cdot \bm{n}\rangle_{\mathcal{F}_h^{\partial,+}}\ge 0$, which holds true because $\bm{\beta}\cdot\bm{n}\ge 0$ on $\mathcal{F}_h^{\partial,+}$.

Motivated by the energy equality, we introduce the energy norm $\triplenorm{\cdot}$ defined as follows:
\begin{equation} \label{eq:energy-norm}
	\begin{aligned}
	\triplenorm{(\bm{r},\bm{v},\bm{\mu})}:=& \bigg(\|\varepsilon^{-1/2}\bm{r}\|_{\mathcal{T}_h}^2+\|\bm{v}\|_{\mathcal{T}_h}^2+\||\tau_t-\frac{1}{2}\bm{\beta}\cdot\bm{n}|^{1/2}(\bm{v}-\bm{\mu})\times\bm{n}\|_{\partial\mathcal{T}_h}^2\\
	&+\||\tau_n-\frac{1}{2}\bm{\beta}\cdot\bm{n}|^{1/2}(\bm{v}-\bm{\mu})\cdot\bm{n}\|_{\partial\mathcal{T}_h}^2\bigg)^{1/2},
	\end{aligned}
\end{equation}
where $\|\cdot\|_D$ denotes the standard $L^2$ norm in the domain $D$. It is readily seen that $\triplenorm{\cdot}$ is indeed a norm under the requirements \eqref{eq:asstau}.

\begin{remark}
	Considering the assumption of a degenerate reaction matrix $\mathcal{R}$, the standard energy equality only offers a degenerate control over the $L^2$ norm of $\bm{u}_h$. Hence, additional arguments are necessary to establish stability.
\end{remark}

\section{Convergence analysis} \label{sec:convergence}
In this section, we will demonstrate the inf-sup stability of our previously mentioned scheme using a weight function and provide the corresponding a priori error estimate. Our primary focus lies in the stability of the proposed scheme as $\varepsilon\rightarrow0$. Thus, we will limit our attention to the scenario where $0< \varepsilon\le 1$. By doing so, the constants in the subsequent results are independent of both the diffusion coefficient $\varepsilon$ and the mesh size $h$.

\subsection{Inf-sup stability}
We shall discuss the inf-sup stability of the bilinear form. It should be noted that when the boundary data $\bm{g}$ is non-zero, we can decompose $\widehat{\bm u}_h$ into two parts: one related to the boundary conditions and the other related to homogeneous boundary. The former can be included in the right-hand side of the problem and is thus independent of the bilinear form. Therefore, we only need to consider the case when $\bm{g}=\bm{0}$.

\subsubsection{Weight function}
We introduce the weight function $\varphi$
\begin{equation}\label{eq:varphi}
\varphi:=e^{-\psi}+\kappa,
\end{equation}
where $\kappa>0$ is a positive constant to be determined later. We can then give the following weighted coecivity using the weight function $\varphi$.
\begin{lemma}[weighted coecivity]
	Let $\varphi$ be given in \eqref{eq:varphi} with $\kappa \ge 1+2b_0^{-1}\|e^{-\psi}\|_{L^\infty(\Omega)}\|\nabla\psi\|^2_{L^{\infty}(\Omega)}$ and $\tau_t,\tau_n$ satisfy \eqref{eq:asstau}. Then the following inequality holds for all $(\bm{w}_h,\bm{u}_h,\bm{\lambda}_h)\in \bm{W}_h\times\bm{V}_h\times\bm{M}_h(\bm{0})$,
	\begin{equation} \label{eq:weightcoe}
			B\big((\bm{w}_h,\bm{u}_h,\bm{\lambda}_h),(\bm{w}_h\varphi,\bm{u}_h\varphi,\bm{\lambda}_h\varphi)\big)\ge C_1\triplenorm{(\bm{w}_h,\bm{u}_h,\bm{\lambda}_h)}^2,
	\end{equation}
for some positive constant $C_1$ independent of $\varepsilon$ and $h$.
\end{lemma}
\begin{proof}
	For any $(\bm{w}_h,\bm{u}_h,\bm{\lambda}_h)\in \bm{W}_h\times\bm{V}_h\times\bm{M}_h(\bm{0})$, we have
\begin{equation}
	B\big((\bm{w}_h,\bm{u}_h,\bm{\lambda}_h),(\bm{w}_h\varphi,\bm{u}_h\varphi,\bm{\lambda}_h\varphi)\big)=(\varepsilon^{-1}\bm{w}_h,\bm{w}_h\varphi)_{\mathcal{T}_h}+T_1+T_2+T_3,
\end{equation}
with
$$
\begin{aligned}
	T_1&:=-(\bm{u}_h,\nabla\times(\bm{w}_h\varphi))_{\mathcal{T}_h}+\langle\bm{\lambda}_h,\bm{n}\times(\bm{w}_h\varphi)\rangle_{\partial\mathcal{T}_h}+(\bm{w}_h,\nabla\times(\bm{u}_h\varphi))_{\mathcal{T}_h}\\
	&~~~~+\langle\bm{n}\times\bm{w}_h,\bm{u}_h\varphi-\bm{\lambda}_h\varphi\rangle_{\partial\mathcal{T}_h},\\
	T_2 &:= (\bm{u}_h,\mathcal{L}_{\bm{\beta}}(\bm{u}_h\varphi)+\gamma\bm{u}_h\varphi)_{\mathcal{T}_h}+\langle\bm{\beta}\cdot\bm{n},\bm{\lambda}_h\cdot\bm{u}_h\varphi\rangle_{\partial\mathcal{T}_h},\\
	T_3 &:= \langle\tau_t(\bm{u}_h-\bm{\lambda}_h)\times\bm{n},\varphi(\bm{u}_h-\bm{\lambda}_h)\times\bm{n}\rangle_{\partial\mathcal{T}_h}\\
	&~~~~+\langle\tau_n(\bm{u}_h-\bm{\lambda}_h)\cdot\bm{n},\varphi(\bm{u}_h-\bm{\lambda}_h)\cdot\bm{n}\rangle_{\partial\mathcal{T}_h},
\end{aligned}
$$
where we use $-\langle (\bm{\beta}\cdot\bm{n})\bm{\lambda}_h,\bm{\lambda}_h\varphi\rangle_{\partial\mathcal{T}_h}+\langle (\bm{\beta}\cdot\bm{n})\bm{\lambda}_h\cdot\bm{n}, \bm{\lambda}_h\cdot \bm{n}\varphi\rangle_{\mathcal{F}_h^{\partial,+}}=0$ since $\bm{\lambda}_h$ is single-valued on the interior facets and $\bm{n}\times\bm{\lambda}_h +\chi_{\Gamma^-}(\bm{\lambda}_h\cdot\bm{n})\bm{n}=0$ on $\partial\Omega$.

By integration by parts, we obtain
$$
\begin{aligned}
T_1 &=-(\bm{u}_h,\nabla\times(\bm{w}_h\varphi))_{\mathcal{T}_h}+\langle\bm{u}_h,\bm{n}\times\bm{w}_h\varphi\rangle_{\partial\mathcal{T}_h}+(\bm{w}_h,\nabla\times(\bm{u}_h\varphi))_{\mathcal{T}_h}\\
&= (\bm{w}_h,\bm{u}_h\times \nabla \varphi)_{\mathcal{T}_h} =-(\bm{w}_h,e^{-\psi}\nabla\psi\times\bm{u}_h)_{\mathcal{T}_h},
\end{aligned}
$$
and 
$$
	\begin{aligned}
T_2=&~\frac{1}{2}(L_{\bm{\beta}}\bm{u}_h+\mathcal{L}_{\bm{\beta}}\bm{u}_h\,\bm{u}_h\varphi)_{\mathcal{T}_h}  + (\gamma \bm{u}_h,\bm{u}_h\varphi)_{\mathcal{T}_h}+\frac{1}{2}(\nabla \varphi\times(\bm{\beta}\times\bm{u}_h),\bm{u}_h)_{\mathcal{T}_h}\\
&~ -\frac{1}{2}(\nabla\varphi\cdot\bm{u}_h,\bm{\beta}\cdot\bm{u}_h)_{\mathcal{T}_h}-\frac{1}{2}\langle \bm{\beta}\cdot\bm{n},\bm{u}_h\cdot\bm{u}_h\varphi\rangle_{\partial\mathcal{T}_h}+\langle\bm{\beta}\cdot\bm{n},\bm{\lambda}_h\cdot\bm{u}_h\varphi\rangle_{\partial\mathcal{T}_h}\\
=&~\frac{1}{2}(\mathcal{R}\bm{u}_h,\bm{u}_h\varphi)_{\mathcal{T}_h}-\frac{1}{2}(\bm{\beta}\cdot\nabla\varphi,\bm{u}_h\cdot\bm{u}_h)_{\mathcal{T}_h}+\frac{1}{2} \langle (\bm{\beta} \cdot\bm{n})\bm{\lambda}_h \cdot\bm{n},
\bm{\lambda}_h\cdot \bm{n}\varphi\rangle_{\mathcal{F}_h^{\partial,+}}\\
&~-\frac{1}{2}\langle \bm{\beta}\cdot\bm{n}(\bm{u}_h-\bm{\lambda}_h),(\bm{u}_h-\bm{\lambda}_h)\varphi\rangle_{\partial\mathcal{T}_h}.
\end{aligned}
$$
Here, we utilize the fact that $\langle(\bm{\beta}\cdot\bm{n}) \bm{\lambda}_h,\bm{\lambda}_h\varphi\rangle_{\partial\mathcal{T}_h}=\langle(\bm{\beta}\cdot\bm{n})\bm{\lambda}_h\cdot\bm{n},\bm{\lambda}_h\cdot\bm{n}\varphi\rangle_{\mathcal{F}^{\partial,+}_h}$. 
Collecting $T_1,T_2$ and $T_3$, then we have
$$
\begin{aligned}
	B\big(&(\bm{w}_h,\bm{u}_h,\bm{\lambda}_h),(\bm{w}_h\varphi,\bm{u}_h\varphi,\bm{\lambda}_h\varphi)\big)\\=&~(\varepsilon^{-1}\bm{w}_h,\bm{w}_h\varphi)_{\mathcal{T}_h}-(\bm{w}_h,e^{-\psi}\nabla\psi\times\bm{u}_h)_{\mathcal{T}_h}+\frac{1}{2}(\mathcal{R}\bm{u}_h,\bm{u}_h\varphi)_{\mathcal{T}_h}\\
	&~+\langle(\tau_t-\frac{1}{2}\bm{\beta}\cdot\bm{n})(\bm{u}_h-\bm{\lambda}_h)\times\bm{n},\varphi(\bm{u}_h-\bm{\lambda}_h)\times\bm{n}\rangle_{\partial\mathcal{T}_h}\\
	&~+\langle(\tau_n-\frac{1}{2}\bm{\beta}\cdot\bm{n})(\bm{u}_h-\bm{\lambda}_h)\cdot\bm{n},\varphi(\bm{u}_h-\bm{\lambda}_h)\cdot\bm{n}\rangle_{\partial\mathcal{T}_h}\\
	&~-\frac{1}{2}(\bm{\beta}\cdot\nabla\varphi,\bm{u}_h\cdot\bm{u}_h)_{\mathcal{T}_h}+\frac{1}{2} \langle (\bm{\beta} \cdot\bm{n})\bm{\lambda}_h \cdot\bm{n},
	\bm{\lambda}_h\cdot \bm{n}\varphi\rangle_{\mathcal{F}_h^{\partial,+}}.
\end{aligned}
$$
With assumption \ref{ass:beta}, we have $-\bm{\beta}\cdot \nabla\varphi = \bm{\beta}\cdot \nabla\psi e^{-\psi}\ge b_0e^{-\psi}$, which insures that
\begin{equation} \label{eq:weighted-bd1}
\begin{aligned}
	& B\big((\bm{w}_h,\bm{u}_h,\bm{\lambda}_h),(\bm{w}_h\varphi,\bm{u}_h\varphi,\bm{\lambda}_h\varphi)\big)\\
	\ge &~ \kappa(\varepsilon^{-1}\bm{w}_h,\bm{w}_h)_{\mathcal{T}_h}-(\bm{w}_h,e^{-\psi}\nabla\psi\times\bm{u}_h)_{\mathcal{T}_h}+\frac{b_0}{2}(\bm{u}_h,e^{-\psi}\bm{u}_h)_{\mathcal{T}_h}\\
	&~+\kappa\langle(\tau_t-\frac{1}{2}\bm{\beta}\cdot\bm{n})(\bm{u}_h-\bm{\lambda}_h)\times\bm{n},(\bm{u}_h-\bm{\lambda}_h)\times\bm{n}\rangle_{\partial\mathcal{T}_h}\\
	&~+\kappa\langle(\tau_n-\frac{1}{2}\bm{\beta}\cdot\bm{n})(\bm{u}_h-\bm{\lambda}_h)\cdot\bm{n},(\bm{u}_h-\bm{\lambda}_h)\cdot\bm{n}\rangle_{\partial\mathcal{T}_h}.\\
\end{aligned}
\end{equation}
Using the Cauchy-Schwarz inequality, we have
\begin{equation} \label{eq:weighted-bd2}
	(\bm{w}_h,e^{-\psi}\nabla\varphi\times\bm{u}_h)_{\mathcal{T}_h} \le \frac{1}{2}\big[\frac{2\|\nabla\psi\|^2_{L^{\infty(\Omega)}}}{b_0}(e^{-\psi}\bm{w}_h,\bm{w}_h)_{\mathcal{T}_h}+\frac{b_0}{2}(e^{-\psi}\bm{u}_h,\bm{u}_h)_{\mathcal{T}_h}\big].
\end{equation}

Therefore, by imposing the condition $\kappa \ge 1+2b_0^{-1}\|e^{-\psi}\|_{L^\infty(\Omega)}\|\nabla\psi\|^2_{L^{\infty}(\Omega)}$ and considering $0<\varepsilon\le 1$, we can deduce from \eqref{eq:weighted-bd1} and \eqref{eq:weighted-bd2} that
\begin{equation}
B\big((\bm{w}_h,\bm{u}_h,\bm{\lambda}_h),(\bm{w}_h\varphi,\bm{u}_h\varphi,\bm{\lambda}_h\varphi)\big)\ge C_1\triplenorm{(\bm{w}_h,\bm{u}_h,\bm{\lambda}_h)}^2,
\end{equation}
with $C_1$ that depends on $\psi,\kappa$, and $b_0$, but is independent of $\varepsilon$ and $h$.
\qed
\end{proof}
\subsubsection{Projections}
Let $\bm{\Pi}_h^{\bm{W}}$ and $\bm{\Pi}_h^{\bm{V}}$  denote the
orthogonal projection ($L^2$-projection) from $\bm{L}^2(\mathcal{T}_h)$ onto $\bm{W}_h$ and $\bm{V}_h$ respectively, and $\bm{P}_{\bm{M}}$ be the $L^2$-projection onto $\bm{M}_h$. Firstly, for the projections we have the following estimate for the difference between $\bm{u}_h\varphi$ and $\bm{\Pi}(\bm{u}_h\varphi)$ where $\bm{\Pi}$ can be $\bm{\Pi}_h^{\bm{W}}, \bm{\Pi}_h^{\bm{V}}$ or $\bm{P}_{\bm{M}}$ with $\bm{u}_h\in \bm{W}_h,\bm{V}_h$ or $\bm{M}_h$, respectively.

\begin{lemma}[superconvergence]
	Let $\varphi\in W^{k+1,\infty}(\Omega)$ be the function introduced in \eqref{eq:varphi}. Then, for $(\bm{r}_h, \bm{v}_h) \in\bm{W}_h\times  \bm{V}_h$ and $\kappa\in\mathbb{R}$, it holds that
\begin{equation}\label{eq:projection}
	\begin{aligned}
		&\left\|{\bm{\Pi}}_h^{\bm{W}}(\varphi \bm{r}_h)-\varphi  \bm{r}_h\right\|_{T} \leq C h_T\|e^{-\psi}\|_{W^{k+1, \infty}(T)}\|\bm{r}_h\|_{T},\\
		&\left\|{\bm{\Pi}}_h^{\bm{W}}(\varphi  \bm{r}_h)-\varphi  \bm{r}_h\right\|_{F} \leq C h_T^{1/2}\|e^{-\psi}\|_{W^{k+1, \infty}(T)}\| \bm{r}_h\|_{T}, \quad \forall F \in \partial T, \\
		&\left\|{\bm{\Pi}}_h^{\bm{V}}(\varphi \bm{v}_h)-\varphi  \bm{v}_h\right\|_{T} \leq C h_T\|e^{-\psi}\|_{W^{k+1, \infty(T)}}\| \bm{v}_h\|_{T},\\
		&\left\|{\bm{\Pi}}_h^{\bm{V}}(\varphi  \bm{v}_h)-\varphi  \bm{v}_h\right\|_{F} \leq C h_T^{1 / 2}\|e^{-\psi}\|_{W^{k+1, \infty}(T)}\| \bm{v}_h\|_{T}, \quad \forall F \in \partial T, 
	\end{aligned}
\end{equation}
where the constant $C$ depends only on $k$, shape-regularity constant, and $\Omega$.
\end{lemma}
\begin{proof}
The lemma can be easily proved by employing the approximation property of the $L^2$ projection onto $\bm{\mathcal{P}}_k$, along with the trace inequality and inverse inequality. A comprehensive proof can be found in the reference \cite{ayuso2009discontinuous}, where all the details are thoroughly presented. \qed
\end{proof}

\begin{lemma}[projection stability] Let $\varphi\in W^{k+1,\infty}(\Omega)$ be the function introduced in \eqref{eq:varphi} and $\tau_t,\tau_n$ satisfy \eqref{eq:asstau}. Then, for $(\bm{r}_h, \bm{v}_h, \bm{\mu}_h) \in\bm{W}_h\times  \bm{V}_h \times \bm{M}_h(\bm{0})$, it holds that
\begin{equation}\label{eq:bound}
\triplenorm{(\bm{\Pi}_h^{\bm{W}}(\bm{r}_h\varphi),\bm{\Pi}_h^{\bm{V}}(\bm{v}_h\varphi),\bm{P}_{\bm{M}}({\bm{\mu}}_h\varphi))}
\le C\triplenorm{(\bm{r}_h,\bm{v}_h,\bm{\mu}_h)}.
\end{equation}
Here, the hidden constant is independent of $\varepsilon$ and $h$.
\end{lemma}
\begin{proof}
By revoking the norm in (3.3), we observe that the norm is weak on the element boundary, where the terms involve only $\bm{v}_h - \bm{\mu}_h$. Consequently, we focus solely on evaluating the stability of these two terms, given that the other terms are standard due to the aforementioned superconvergence property. For clarity, we illustrate the stability analysis using the tangential component as an example, noting that the analysis for the normal component follows a similar approach.

Using the superconvergence property \eqref{eq:projection}, we have
$$ 
\begin{aligned}
&\||\tau_t-\frac{1}{2}\bm{\beta}\cdot\bm{n}|^{1/2}\bm{n} \times (\bm{\Pi}_h^{\bm V}(\bm{v}_h\varphi) - \bm{P}_{\bm M}(\bm{\mu}_h\varphi) )\|_{\partial\mathcal{T}_h} \\ 
\le &~ \||\tau_t-\frac{1}{2}\bm{\beta}\cdot\bm{n}|^{1/2}\bm{n} \times (\bm{\Pi}_h^{\bm V}(\bm{v}_h\varphi) - \bm{P}_M(\bm{v}_h \varphi) - \bm{P}_{\bm M}(\bm{\mu}_h\varphi  - \bm{v}_h\varphi)) \|_{\partial\mathcal{T}_h} \\
\le &~ Ch^{1/2} \|\bm{v}_h\|_{\mathcal{T}_h} + \||\tau_t-\frac{1}{2}\bm{\beta}\cdot\bm{n}|^{1/2} \bm{n} \times \bm{P}_{\bm M}(\bm{\mu}_h\varphi  - \bm{v}_h\varphi) \|_{\partial\mathcal{T}_h}.
\end{aligned}
$$ 
With reference to \eqref{eq:min-max-equiv}, we can bound the last term by $\| |\tau_t-\frac{1}{2}\bm{\beta}\cdot\bm{n}|^{1/2} \bm{n} \times (\bm{\mu}_h - 
\bm{v}_h) \|_{\partial\mathcal{T}_h}$ in the energy norm, namely
$$
\begin{aligned}
\||\tau_t- & \frac{1}{2}\bm{\beta}\cdot\bm{n}|^{1/2} \bm{n} \times \bm{P}_{\bm M}(\bm{\mu}_h\varphi  - \bm{v}_h\varphi)\|_{\partial\mathcal{T}_h} \\
\le &~\| (\max_F|\tau_t-  \frac{1}{2}\bm{\beta}\cdot\bm{n}|^{1/2}) \bm{n} \times \bm{P}_{\bm M}(\bm{\mu}_h\varphi  - \bm{v}_h\varphi)\|_{\partial\mathcal{T}_h} \\
\le &~\| (\max_F|\tau_t-  \frac{1}{2}\bm{\beta}\cdot\bm{n}|^{1/2}) \bm{n} \times (\bm{\mu}_h\varphi  - \bm{v}_h\varphi)\|_{\partial\mathcal{T}_h} \\
\le &~C \| |\tau_t-  \frac{1}{2}\bm{\beta}\cdot\bm{n}|^{1/2} \bm{n} \times (\bm{\mu}_h  - \bm{v}_h)\|_{\partial\mathcal{T}_h}.
\end{aligned}
$$ 
This completes the proof. \qed 
\end{proof}

\subsection{Stability}
We are now able to provide the following stability result.
\begin{theorem}[inf-sup stability]
	Let $\tau_t$ and $\tau_n$ satisfy \eqref{eq:asstau}, then there exist positive constants $\alpha$ and $h_0$ independent of $h$ and $\varepsilon$, such that for $h<h_0$,
	\begin{equation}\label{eq:infsup}
	\sup_{0\neq(\bm{r}_h,\bm{v}_h,\bm{\mu}_h)\in\bm{W}_h\times\bm{V}_h\times\bm{M}_h(\bm{0})} \frac{B\big((\bm{w}_h,\bm{u}_h,\bm{\lambda}_h),(\bm{r}_h,\bm{v}_h,\bm{\mu}_h)\big)}{\triplenorm{(\bm{r}_h,\bm{v}_h,\bm{\mu}_h)}} \ge \alpha \triplenorm{(\bm{w}_h,\bm{u}_h,\bm{\lambda}_h)},
\end{equation}
for all $(\bm{w}_h,\bm{u}_h,\bm{\lambda}_h)\in\bm{W}_h\times\bm{V}_h\times\bm{M}_h(\bm{0})$.
\end{theorem}
\begin{proof}
For simplicity, we denote $\bm{\delta}\bm{w}_h^\varphi = \bm{w}_h\varphi-\bm{\Pi}_h^{\bm{W}}(\bm{w}_h\varphi)$, $\bm{\delta}\bm{u}_h^\varphi = \bm{u}_h\varphi-\bm{\Pi}_h^{\bm{V}}(\bm{u}_h\varphi)$  and $\bm{\delta}\bm{\lambda}_h^\varphi=\bm{\lambda}_h\varphi-\bm{P}_{\bm{M}}(\bm{\lambda}_h\varphi)$. 
By integrating by parts, we have
$$
\begin{aligned}
	B\big(&(\bm{w}_h,\bm{u}_h,\bm{\lambda}_h),(\bm{\delta}\bm{w}_h^\varphi,\bm{\delta}\bm{u}_h^\varphi,\bm{\delta}\bm{\lambda}_h^\varphi)\big) \\
	= &~(\varepsilon^{-1}\bm{w}_h,\bm{\delta}\bm{w}_h^\varphi)_{\mathcal{T}_h}-(\nabla\times\bm{u}_h,\bm{\delta}\bm{w}_h^\varphi)_{\mathcal{T}_h}+\langle\bm{\lambda}_h-\bm{u}_h,\bm{n}\times\bm{\delta}\bm{w}_h^\varphi\rangle_{\partial\mathcal{T}_h}\\
	&~+(\nabla\times\bm{w}_h,\bm{\delta}\bm{u}_h^\varphi)_{\mathcal{T}_h}+(L_{\bm{\beta}}\bm{u}_h+\gamma\bm{u}_h,\bm{\delta}\bm{u}_h^\varphi)_{\mathcal{T}_h}\\
	&~+\langle(\tau_t-\bm{\beta}\cdot\bm{n})(\bm{u}_h-\bm{\lambda}_h)\times\bm{n},\bm{\delta}\bm{u}_h^\varphi\times\bm{n}\rangle_{\partial\mathcal{T}_h}\\
	&~+\langle(\tau_n-\bm{\beta}\cdot\bm{n})(\bm{u}_h-\bm{\lambda}_h)\cdot\bm{n},\bm{\delta}\bm{u}_h^\varphi\cdot\bm{n}\rangle_{\partial\mathcal{T}_h}\\
	&~-\langle \bm{n}\times\bm{w}_h+\tau_t\bm{n}\times(\bm{u}_h-\bm{\lambda}_h)\times\bm{n}+\tau_n \bm{n}\cdot(\bm{u}_h-\bm{\lambda}_h)\bm{n},\bm{\delta}\bm{\lambda}_h^\varphi\rangle_{\partial\mathcal{T}_h}.
\end{aligned}
$$
Due to the orthogonality of the projections, we have
$$
\begin{aligned}
\langle \bm{n}\times\bm{w}_h+\tau_t\bm{n}\times(\bm{u}_h-\bm{\lambda}_h)\times\bm{n}+\tau_n \bm{n}\cdot(\bm{u}_h-\bm{\lambda}_h)\bm{n},\bm{\delta}\bm{\lambda}_h^\varphi\rangle_{\partial\mathcal{T}_h}&=0,\\
	(\nabla\times\bm{u}_h,\bm{\delta}\bm{w}_h^\varphi)_{\mathcal{T}_h}&=0,\\
	(\nabla\times\bm{w}_h,\bm{\delta}\bm{u}_h^\varphi)_{\mathcal{T}_h}&=0.
\end{aligned}
$$

We estimate the remaining terms as follows. 
Firstly, by considering the requirements \eqref{eq:asstau2}-\eqref{eq:asstau3} for $\tau_t$ and $\tau_n$, together with the superconvergent result \eqref{eq:projection}, we have
$$
\begin{aligned}
	(\varepsilon^{-1}\bm{w}_h,&\bm{\delta}\bm{w}_h^\varphi)_{\mathcal{T}_h}\le Ch\|\varepsilon^{-1/2}\bm{w}_h\|_{\mathcal{T}_h}^2,\\
	\langle(\tau_t-\bm{\beta}&\cdot\bm{n})(\bm{u}_h-\bm{\lambda}_h)\times\bm{n},\bm{\delta}\bm{u}_h^\varphi\times\bm{n}\rangle_{\partial\mathcal{T}_h}\\
	&\leq \left\| \big(|\tau_t-\frac{1}{2}\bm{\beta}\cdot\bm{n}| + \frac{1}{2}| \bm{\beta}\cdot\bm{n}|\big)(\bm{u}_h-\bm{\lambda}_h)\times\bm{n}\right\|_{\partial \mathcal{T}_h}  \|\bm{\delta}\bm{u}_h^\varphi\times\bm{n}\|_{\partial\mathcal{T}_h}\\
	&\le Ch^{1/2}\||\tau_t-\frac{1}{2}\bm{\beta}\cdot\bm{n}|^{1/2}(\bm{u}_h-\bm{\lambda}_h)\times\bm{n}\|_{\partial \mathcal{T}_h}\|\bm{u}_h\|_{\mathcal{T}_h},\\
	\langle(\tau_n-\bm{\beta}&\cdot\bm{n})(\bm{u}-\bm{\lambda}_h)\cdot\bm{n},\bm{\delta}\bm{u}_h^\varphi\cdot\bm{n}\rangle_{\partial\mathcal{T}_h}\\&
	\le \left\| \big(|\tau_n-\frac{1}{2}\bm{\beta}\cdot\bm{n}| + \frac{1}{2}|\bm{\beta}\cdot\bm{n}| \big) (\bm{u}_h-\bm{\lambda}_h)\cdot\bm{n}\right\|_{\partial \mathcal{T}_h} \|\bm{\delta}\bm{u}_h^\varphi\cdot\bm{n}\|_{\partial\mathcal{T}_h}\\
	&\le Ch^{1/2}\||\tau_n-\frac{1}{2}\bm{\beta}\cdot\bm{n}|^{1/2}(\bm{u}_h-\bm{\lambda}_h)\cdot\bm{n}\|_{\partial \mathcal{T}_h}\|\bm{u}_h\|_{\mathcal{T}_h}.
\end{aligned}
$$
Secondly, taking into account the orthogonality of $\bm{\Pi}_h^{\bm{V}}$ on $\bm{\mathcal{P}}_{k-1}$ and utilizing the inverse inequality, we can deduce that 
$$
	\begin{aligned}
		(L_{\bm{\beta}}&\bm{u}_h+\gamma\bm{u}_h,\bm{\delta}\bm{u}_h^\varphi)_{\mathcal{T}_h}\\
		&=(-\bm{\beta}\times(\nabla\times\bm{u}_h)+(\nabla\bm{u}_h)^T\bm{\beta}+(\nabla\bm{\beta})^T\bm{u}_h+\gamma\bm{u}_h,\bm{\delta}\bm{u}_h^\varphi)_{\mathcal{T}_h}\\
	&=(-(\bm{\beta}-\bar{\bm{\beta}})\times(\nabla_h\times\bm{u}_h)+(\nabla\bm{u}_h)^T(\bm{\beta}-\bar{\bm{\beta}})+(\nabla\bm{\beta})^T\bm{u}_h+\gamma\bm{u}_h,\bm{\delta}\bm{u}_h^\varphi)_{\mathcal{T}_h}\\
	&\le Ch\|\bm{u}_h\|_{\mathcal{T}_h}^2,
	\end{aligned}
$$
where $\bar{\bm{\beta}}$ is the piecewise constant approximation for $\bm{\beta}$.

For the term $\langle\bm{\lambda}_h-\bm{u}_h,\bm{n}\times\bm{\delta}\bm{w}_h^\varphi\rangle_{\partial\mathcal{T}_h}$, invoking the requirement \eqref{eq:asstau4}, we have
$$
\begin{aligned}
	\langle\bm{\lambda}_h&-\bm{u}_h,\bm{n}\times\bm{\delta}\bm{w}_h^\varphi\rangle_{\partial\mathcal{T}_h}\\&\le \|\tau_t^{1/2}(\bm{u}_h-\bm{\lambda}_h)\times\bm{n}\|_{\partial\mathcal{T}_h}\|\tau_t^{-1/2}\bm{\delta}\bm{w}_h^\varphi\|_{\partial\mathcal{T}_h}\\
	&\le C(\frac{\varepsilon}{\tau_t})^{1/2}\||\tau_t-\frac{1}{2}\bm{\beta}\cdot\bm{n}|^{1/2}(\bm{u}_h-\bm{\lambda}_h)\times\bm{n}\|_{\partial\mathcal{T}_h}\|\varepsilon^{-1/2}\bm{\delta}\bm{w}_h^\varphi\|_{\partial\mathcal{T}_h}\\
	&\le C(\frac{\varepsilon h}{\tau_t})^{1/2}\||\tau_t-\frac{1}{2}\bm{\beta}\cdot\bm{n}|^{1/2}(\bm{u}_h-\bm{\lambda}_h)\times\bm{n}\|_{\partial\mathcal{T}_h}\|\varepsilon^{-1/2}\bm{w}_h\|_{\mathcal{T}_h}\\
	&\le C(h^2+\varepsilon h)^{1/2}\||\tau_t-\frac{1}{2}\bm{\beta}\cdot\bm{n}|^{1/2}(\bm{u}_h-\bm{\lambda}_h)\times\bm{n}\|_{\partial\mathcal{T}_h}\|\varepsilon^{-1/2}\bm{w}_h\|_{\mathcal{T}_h}.
\end{aligned}
$$
By combining all the aforementioned terms, we have
$$
B\big((\bm{w}_h,\bm{u}_h,\bm{\lambda}_h),(\bm{\delta}\bm{w}_h^\varphi,\bm{\delta}\bm{u}_h^\varphi,\bm{\delta}\bm{\lambda}_h^\varphi)\big)\le Ch^{1/2}\triplenorm{(\bm{w}_h,\bm{u}_h,\bm{\lambda}_h)}^2,
$$
for some $C$ independent of $\varepsilon$ and $h$. As a result, there exists a value $h_0>0$ such that for $h<h_0$, it  holds that
$$
B\big((\bm{w}_h,\bm{u}_h,\bm{\lambda}_h),(\bm{\delta}\bm{w}_h^\varphi,\bm{\delta}\bm{u}_h^\varphi,\bm{\delta}\bm{\lambda}_h^\varphi)\big)\le \frac{C_1}{2}\triplenorm{(\bm{w}_h,\bm{u}_h,\bm{\lambda}_h)}^2.
$$
Invoking the weighted coecivity result \eqref{eq:weightcoe},  we obtain
\begin{equation}\label{eq:coecivity}
	B\big((\bm{w}_h,\bm{u}_h,\bm{\lambda}_h),(\bm{\Pi}_h^{\bm{W}}(\bm{w}_h\varphi),\bm{\Pi}_h^{\bm{V}}(\bm{u}_h\varphi),\bm{P}_{\bm{M}}({\bm{\lambda}}_h\varphi)\big)\ge \frac{C_1}{2}\triplenorm{(\bm{w}_h,\bm{u}_h,\bm{\lambda}_h)}^2.
\end{equation}
Finally, we get the desired result by combining \eqref{eq:coecivity} and \eqref{eq:bound}.
\qed
\end{proof}

\subsection{Error estimates}
The standard approximation properties of projections imply
\begin{equation}\label{eq:projerrT}
\begin{aligned}
	\|	\bm{\Pi}_h^{\bm{W}} \bm{w}-\bm{w}\|_{\bm{H}^t(T)}&\le Ch^{s+1-t}|\bm{w}|_{\bm{H}^{s+1}(T)},\quad \forall T\in\mathcal{T}_h,\\
	\|	\bm{\Pi}_h^{\bm{V}} \bm{u}-\bm{u}\|_{\bm{H}^t(T)}&\le Ch^{s+1-t}|\bm{u}|_{\bm{H}^{s+1}(T)},\quad \forall T\in\mathcal{T}_h,\\
	\|\bm{P}_{\bm{M}}\bm{u}-\bm{u}\|_{L^2(F)}&\le Ch^{s+1/2}|\bm{u}|_{\bm{H}^{s+1}(T)},\quad \forall F\in \partial T,\ T\in\mathcal{T}_h,
\end{aligned}
\end{equation}
for $t = 0,1$ and $s\in[0,k]$. We introduce the following notation
$$
\begin{aligned}
	\bm{\varepsilon}_h^{\bm{w}} &:=\boldsymbol{w}_h-\bm{\Pi}_h^{\bm{W}} \boldsymbol{w},  &\bm{\delta} \bm{w}&:=\boldsymbol{w}-\bm{\Pi}_h^{\bm{W}} \boldsymbol{w}, \\
	\bm{\varepsilon}_h^{\bm{u}} &:=\bm{u}_h-\bm{\Pi}_h^{\bm{V}} \bm{u},  &\bm{\delta} \bm{u}&:=\bm{u}-\bm{\Pi}_h^{\bm{V}} \bm{u}, \\
	\bm{\varepsilon}_h^{\widehat{\bm{u}}} &:=\widehat{\bm{u}}_h-\bm{P}_{\bm{M}} u,  &{\bm{\delta} \widehat{\bm{u}}}&:=\bm{u}-\bm{P}_{\bm{M}} \bm{u} .
\end{aligned}
$$
With the Galerkin orthogonality, we have the following error equation:
\begin{equation} \label{eq:error-equ}
	\begin{aligned}
		B\big((\bm{\varepsilon}_h^{\bm{w}},&\bm{\varepsilon}_h^{\bm{u}} ,\bm{\varepsilon}_h^{\widehat{\bm{u}}}),(\bm{r},\bm{v},\bm{\mu})\big) = B\big((\bm{\delta}\bm{w},\bm{\delta}\bm{u},\bm{\delta}\widehat{\bm{u}}),(\bm{r},\bm{v},\bm{\mu})\big)\\
		=&~(\varepsilon^{-1}\bm{\delta}\bm{w},\bm{r})_{\mathcal{T}_h}-(\bm{\delta}\bm{u},\nabla\times\bm{r})_{\mathcal{T}_h}+\langle\bm{\delta}\widehat{\bm{u}},\bm{n}\times\bm{r}\rangle_{\partial\mathcal{T}_h}+(\bm{\delta}\bm{u},\mathcal{L}_{\bm{\beta}}\bm{v}+\gamma\bm{v})_{\mathcal{T}_h}\\
		&~+\langle(\bm{\beta}\cdot\bm{n})\bm{\delta}\widehat{\bm{u}},\bm{v}\rangle_{\partial\mathcal{T}_h}	+(\bm{\delta}\bm{w},\nabla\times\bm{v})_{\mathcal{T}_h}+\langle \bm{n}\times\bm{\delta}\bm{w},\bm{v}-\bm{\mu}\rangle_{\partial\mathcal{T}_h}\\
		&~+\langle\tau_t \bm{n}\times(\bm{\delta}\bm{u}-\bm{\delta}\widehat{\bm{u}}),\bm{n}\times(\bm{v}-\bm{\mu})\rangle_{\partial\mathcal{T}_h}\\
		&~+\langle\tau_n \bm{n}\cdot(\bm{\delta}\bm{u}-\bm{\delta}\widehat{\bm{u}}),\bm{n}\cdot(\bm{v}-\bm{\mu})\rangle_{\partial\mathcal{T}_h},
	\end{aligned}
\end{equation}
for all $(\bm{r},\bm{v},\bm{\mu})\in \bm{W}_h\times\bm{V}_h\times\bm{M}_h(\bm{0})$. Now we are ready to give the following error estimate.
\begin{theorem}[error estimate]\label{thm:converge}
Let $(\bm{w},\bm{u})$ be the solution of the problem (\ref{eq:mixed}), and $(\bm{w}_h,\bm{u}_h,\widehat{\bm{u}}_h)$ be the solution of the HDG method (\ref{eq:hdg}) where the stabilization functions $\tau_t$ and $\tau_n$ satisfy (\ref{eq:asstau}). Then there exists $h_0$ independent of $\varepsilon$ such that when $h<h_0$,
	\begin{equation}\label{eq:error}
		\begin{aligned}
		&\triplenorm{(\bm{w}-\bm{w}_h,\bm{u}-\bm{u}_h,\bm{u}-\widehat{\bm{u}}_h)} \\
		\le &~C(h^{1/2}+\varepsilon^{1/2})h^{s_w+1/2}|\varepsilon^{-1/2}\bm{w}|_{\bm{H}^{s_w+1}(\Omega)} +Ch^{s_u+1/2}|\bm{u}|_{\bm{H}^{s_u+1}(\Omega)}\\ 
		= &~ C\varepsilon^{1/2}(h^{1/2}+\varepsilon^{1/2})h^{s_w+1/2}|\nabla\times\bm{u}|_{\bm{H}^{s_w+1}(\Omega)}
		+Ch^{s_u+1/2}|\bm{u}|_{\bm{H}^{s_u+1}(\Omega)},
		\end{aligned}
	\end{equation}
		for all $s_w\in [0,k]$ and $s_u\in[0,k]$. Here, the hidden constant is independent of $\varepsilon$ and $h$.
\end{theorem}
\begin{proof} Firstly, due to the orthogonality of the projections, it holds that
\begin{equation*}
	\begin{aligned}
		(\bm{\delta}\bm{u},\nabla\times\bm{r})_{\mathcal{T}_h}&=0,\\
		\langle\bm{\delta}\widehat{\bm{u}},\bm{n}\times\bm{r}\rangle_{\partial\mathcal{T}_h}&=0,\\
		(\bm{\delta}\bm{w},\nabla\times\bm{v})_{\mathcal{T}_h}&=0,\\
		\langle\tau_t \bm{n}\times\bm{\delta}\widehat{\bm{u}},\bm{n}\times(\bm{v}-\bm{\mu})\rangle_{\partial\mathcal{T}_h}+\langle\tau_n \bm{n}\cdot\bm{\delta}\widehat{\bm{u}},\bm{n}\cdot(\bm{v}-\bm{\mu})\rangle_{\partial\mathcal{T}_h}&=0.
	\end{aligned}
\end{equation*}
The remaining terms in the error equations \eqref{eq:error-equ} lead to 
$$
\begin{aligned}
	B\big((\bm{\delta}\bm{w},&\bm{\delta}\bm{u},\bm{\delta}\widehat{\bm{u}}),(\bm{r},\bm{v},\bm{\mu})\big) \\
	=&~(\varepsilon^{-1}\bm{\delta}\bm{w},\bm{r})_{\mathcal{T}_h}+(\bm{\delta}\bm{u},\mathcal{L}_{\bm{\beta}}\bm{v}+\gamma\bm{v})_{\mathcal{T}_h}+\langle(\bm{\beta}\cdot\bm{n})\bm{\delta}\widehat{\bm{u}},\bm{v}\rangle_{\partial\mathcal{T}_h}\\
	&~+\langle \bm{n}\times\bm{\delta}\bm{w},\bm{v}-\bm{\mu}\rangle_{\partial\mathcal{T}_h}
	+\langle\tau_t \bm{n}\times\bm{\delta}\bm{u},\bm{n}\times(\bm{v}-\bm{\mu})\rangle_{\partial\mathcal{T}_h}\\
	&~+\langle\tau_n \bm{n}\cdot\bm{\delta}\bm{u},\bm{n}\cdot(\bm{v}-\bm{\mu})\rangle_{\partial\mathcal{T}_h}.
\end{aligned}
$$
Once again, considering the piecewise constant approximation $\bar{\bm{\beta}}$ of $\bm{\beta}$, along with the condition \eqref{eq:min-max-equiv} for $\tau_t$ and $\tau_n$, and the approximation properties \eqref{eq:projerrT}, we can establish the following:
$$
\begin{aligned}
	(\varepsilon^{-1}\bm{\delta}\bm{w},\bm{r})_{\mathcal{T}_h}&\le \|\varepsilon^{-1/2}\bm{\delta}\bm{w}\|_{\mathcal{T}_h}\|\varepsilon^{-1/2}\bm{r}\|_{\mathcal{T}_h}\\
	&\le h^{s_w+1}|\varepsilon^{-1/2}\bm{w}|_{\bm{H}^{s_w+1}}\|\varepsilon^{-1/2}\bm{r}\|_{\mathcal{T}_h},\\
	(\bm{\delta}\bm{u},\mathcal{L}_{\bm{\beta}}\bm{v}+\gamma\bm{v})_{\mathcal{T}_h}&=((\nabla\bm{\beta})^T\bm{v}-(\nabla\cdot\bm{\beta})\bm{v}+\gamma\bm{v}-(\nabla_h\bm{v})^T(\bm{\beta}-\bar{\bm{\beta}})\bm{v},\bm{\delta}\bm{u}))\\
	&\le Ch^{s_u+1}|\bm{u}|_{\bm{H}^{s_u+1}}\|\bm{v}\|_{\mathcal{T}_h},\\
	\langle(\bm{\beta}\cdot\bm{n})\bm{\delta}\widehat{\bm{u}},\bm{v}\rangle_{\partial\mathcal{T}_h}&=\langle(\bm{\beta}-\bar{\bm{\beta}})\cdot\bm{n},\bm{\delta}\widehat{\bm{u}}\cdot\bm{v}\rangle_{\partial\mathcal{T}_h} \le Ch^{s_u+1}|\bm{u}|_{\bm{H}^{s_u+1}}\|\bm{v}\|_{\mathcal{T}_h},\\
	\langle\tau_t \bm{n}\times\bm{\delta}\bm{u},\bm{n}\times&(\bm{v}-\bm{\mu})\rangle_{\partial\mathcal{T}_h}\\
	&\le Ch^{s_u+1/2}|\bm{u}|_{\bm{H}^{s_u+1}}\| \tau_t^{1/2}(\bm{v}-\bm{\mu})\times\bm{n}\|_{\partial \mathcal{T}_h},\\
	&\le Ch^{s_u+1/2}|\bm{u}|_{\bm{H}^{s_u+1}}\||\tau_t-\frac{1}{2}\bm{\beta}\cdot\bm{n}|^{1/2}(\bm{v}-\bm{\mu})\times\bm{n}\|_{\partial \mathcal{T}_h},\\
	\langle\tau_n \bm{n}\cdot\bm{\delta}\bm{u},\bm{n}\cdot&(\bm{v}-\bm{\mu})\rangle_{\partial\mathcal{T}_h}\\&\le Ch^{s_u+1/2}|\bm{u}|_{\bm{H}^{s_u+1}}\||\tau_n-\frac{1}{2}\bm{\beta}\cdot\bm{n}|^{1/2}(\bm{v}-\bm{\mu})\cdot\bm{n}\|_{\partial \mathcal{T}_h},
\end{aligned}
$$
for all $s_w\in [0,k]$ and $s_u\in[0,k]$. For the term $\langle \bm{n}\times\bm{\delta}\bm{w},\bm{v}-\bm{\mu}\rangle_{\partial\mathcal{T}_h}$, similarly with the requirement \eqref{eq:asstau4} and condition \eqref{eq:min-max-equiv}, we have
$$
\begin{aligned}
\langle \bm{n}\times&\bm{\delta}\bm{w},\bm{v}-\bm{\mu}\rangle_{\partial\mathcal{T}_h}\\
&\le C(\varepsilon \tau_t^{-1})^{1/2}h^{s_w+1/2}|\varepsilon^{-1/2}\bm{w}|_{\bm{H}^{s_w+1}}\|\tau_t^{1/2}(\bm{v}-\bm{\mu})\times\bm{n}\|_{\partial \mathcal{T}_h} \\
&\le C(h^{1/2}+\varepsilon^{1/2})h^{s_w+1/2}|\varepsilon^{-1/2}\bm{w}|_{\bm{H}^{s_w+1}}\||\tau_t-\frac{1}{2}\bm{\beta}\cdot\bm{n}|^{1/2}(\bm{v}-\bm{\mu})\times\bm{n}\|_{\partial \mathcal{T}_h}.
\end{aligned}
$$
Then, collecting the above results and utilizing the inf-sup stability (\ref{eq:infsup}), we arrive at
\begin{equation}\label{eq:galerkinerr}
\begin{aligned}
\triplenorm{(\bm{\varepsilon}_h^{\bm{w}},\bm{\varepsilon}_h^{\bm{u}} ,\bm{\varepsilon}_h^{\widehat{\bm{u}}})}\le &~C (h+\varepsilon)^{1/2}h^{s_w+1/2}|\varepsilon^{-1/2}\bm{w}|_{\bm{H}^{s_w+1}(\Omega)}\\
&~+Ch^{s_u+1/2}|\bm{u}|_{\bm{H}^{s_u+1}(\Omega)}.
\end{aligned}
\end{equation}
With the approximation property of projections, we readily have
\begin{equation}\label{eq:projecterr}
\triplenorm{(\bm{\delta}\bm{w},\bm{\delta}\bm{u},\bm{\delta}\widehat{\bm{u}})}\le Ch^{s_w+1}|\varepsilon^{-1/2}\bm{w}|_{\bm{H}^{s_w+1}(\Omega)} + Ch^{s_u+1/2}|\bm{u}|_{\bm{H}^{s_u+1}(\Omega)}.
\end{equation}
Note that $\bm{w}=\varepsilon\nabla\times\bm{u}$, combing \eqref{eq:galerkinerr}, \eqref{eq:projecterr} and employing the triangle inequality, we obtain
the desired estimate \eqref{eq:error}. \qed
\end{proof}

\section{Local postprocessing} \label{sec:postprocessing}
In this section, we will introduce some straightforward element-by-element postprocessing techniques that can be utilized to generate new approximations of $\bm{u}$ in both 2D and 3D scenarios.

In the case of translational symmetry, the mixed problem \eqref{eq:mixed} can be reduced to the following two-dimensional (2D) boundary value problem on a domain $\Omega\subset\mathbb{R}^2$ \cite{heumann2013stabilized}:
\begin{equation} \label{eq:Hcurl-cd-2d}
	\left\{
	\begin{aligned}
		w-  \varepsilon \nabla \cdot  ({\rm \mathbf{R}}\bm{u}) &=0\quad \text{in }\Omega,\\
	\mathbf{R}\nabla w - {\mathbf{R}}\bm{\beta} 
		\nabla \cdot ({\rm \mathbf{R}} \bm{u}) + \nabla (\bm{\beta} \cdot \bm{u}) +  \gamma
		\bm{u}&= \bm{f} \quad \text{in }\Omega, \\
		({\mathbf{R}}\bm{n} \cdot\bm{u}){\mathbf{R}}\bm{n} + \chi_{\Gamma^-} (\bm{u} \cdot \bm{n}) \bm{n} &= \bm{g} \quad \text{on } \Gamma, 
	\end{aligned}
	\right.
\end{equation}
with the $\frac{\pi}{2}$-rotation matrix $\mathbf{R}=\left(\begin{matrix}0 & 1 \\ -1 & 0\end{matrix}\right)$. 
The corresponding HDG scheme aims to find $(w_h,\bm{u}_h,\widehat{\bm{u}}_h)\in W_h\times \bm{V}_h\times \bm{M}_h(\bm{g})$ such that:\begin{subequations}
	\label{eq:hdg2d}
	\begin{align}
		(\varepsilon^{-1}w_h,r)_{\mathcal{T}_h}- (\bm{u}_h,\mathbf{R}\nabla r)_{\mathcal{T}_h} +\langle \widehat{\bm{u}}_h,(\mathbf{R}\bm{n})r\rangle_{\partial\mathcal{T}_h}&=0,\\
		\begin{aligned}
		(w_h,\nabla\cdot(\mathbf{R}\bm{v}))+\langle (\mathbf{R}\bm{n})\widehat{w}_h,\bm{v}\rangle_{\partial\mathcal{T}_h} + (\bm{u}_h,\mathcal{L}_{\bm{\beta}}\bm{v}+\gamma \bm{v})_{\mathcal{T}_h}\\+\langle \mathbf{R}\bm{n}\cdot\widehat{\bm{u}}_h,\mathbf{R}\bm{\beta}\cdot\bm{v}\rangle_{\partial\mathcal{T}_h}+
		\langle\widehat{ \bm{\beta}\cdot\bm{u}_h},\bm{n}\cdot\bm{v}\rangle_{\partial\mathcal{T}_h} \end{aligned}&= (\bm{f},\bm{v})_{\mathcal{T}_h},\\
		-\langle (\mathbf{R}\bm{n})\widehat{w}_h+\mathbf{R}\bm{\beta}(\mathbf{R}\bm{n}\cdot\widehat{\bm{u}}_h)+\widehat{ \bm{\beta}\cdot\bm{u}_h}\bm{n},\bm{\mu}\rangle_{\partial\mathcal{T}_h\backslash\mathcal{F}_h^\partial }&=0,\\
		\langle 	({\mathbf{R}}\bm{n} \cdot\bm{u}_h){\mathbf{R}}\bm{n} + \chi_{\Gamma^-} (\bm{u}_h \cdot \bm{n}) \bm{n}  - \chi_{\Gamma^+}[(\bm{u}_h-\widehat{\bm{u}}_h)\cdot\bm{n})]\bm{n},\mu\rangle_{\mathcal{F}_h^\partial } &= \langle\bm{g},\bm{\mu}\rangle_{\mathcal{F}_h^\partial },
	\end{align}
\end{subequations}
for all $(r,\bm{v},\bm{\mu})\in W_h\times \bm{V}_h\times \bm{M}_h(\bm{0})$ where $\bm{M}_h(\bm{g}) :=\{\bm{\mu}\in \bm{M}_h:\langle{(\mathbf{R}\bm{n}\cdot\bm\mu})\mathbf{R}\bm{n}+\chi_{\Gamma^-}(\bm{\mu}\cdot\bm{n})\bm{n},\bm{\xi}\rangle_{\mathcal{F}_h^\partial }=\langle\bm{g},\bm{\xi}\rangle_{\mathcal{F}_h^\partial },\forall\bm{\xi}\in\bm{M}_h\}$, $W_h$ here is a scalar space and the numerical traces are set to be 
\begin{subequations}
	\begin{align}
		\widehat{w}_h &= w_h + \tau_t(\bm{u}_h-\widehat{\bm{u}}_h)\cdot\mathbf{R}\bm{n},\\
		\widehat{ \bm{\beta}\cdot\bm{u}_h} &= \bm{\beta}\cdot\widehat{\bm{u}}_h+\tau_n(\bm{u}_h-\widehat{\bm{u}}_h)\cdot \bm{n},
	\end{align}
\end{subequations}
with the same requirements (\ref{eq:asstau1})-(\ref{eq:asstau4}) for $\tau_t,\tau_n$.

\subsection{An existing postprocessing scheme for 2D problem}
Firstly, we employ the postprocessing scheme introduced in \cite{nguyen2011hybridizable} to determine $\bm{u}_h^*$ as the element of $\bm{\mathcal{P}}_{k+1}(T)$ satisfying the following condition for all $T\in\mathcal{T}_h$:
\begin{subequations}\label{eq:postprocess1}
	\begin{align}
		\langle \mathbf{R}\bm{n}\cdot(\bm{u}_h^*  -\widehat{\bm{u}}_h),\eta\rangle_F &=0,\ &\forall \eta\in \mathcal{P}_{k+1},\forall F\in\partial T,\label{eq:postprocess1-1}\\
		(\bm{u}_h^*-\bm{u}_h,\mathbf{R}\nabla v)_T&=0,\ &\forall v\in \mathcal{P}_k(T),\\
		(\nabla\cdot(\bm{u}_h^*-\bm{u}_h),s)_T &=0,\ &\forall s\in \mathcal{P}_{k-1}(T).
	\end{align}
\end{subequations}

It is evident that the approximation $\bm{u}_h^*$, if it exists, adheres to the $H(\text{curl})$ conforming property. The well-posedness of $\bm{u}_h^*$ was established in \cite[Proposition 5]{nguyen2011hybridizable}, assuming the fulfillment of the acute angle condition. Indeed, the requirement of acuteness can be relaxed as non-obtuseness. 

\begin{lemma}[postprocess property]\label{lemma:post1}
	The postprocessed solution $\bm{u}_h^*$ satisfies
	$$
	(\nabla\cdot(\mathbf{R} \bm{u}_h^*),q)_T=(\varepsilon^{-1}w_h,q)_T,\quad\forall q\in \mathcal{P}_k(T), \ T\in\mathcal{T}_h.
	$$
\end{lemma}
\begin{proof}
	The proof has already been presented in \cite[Lemma 6.1]{nguyen2011hybridizable}. For the sake of clarity, we sketch the proof here.
	From the HDG equation (\ref{eq:hdg2d}), we have
	$$
	(\varepsilon^{-1}w_h,q)_T-(\bm{u}_h,\mathbf{R}\nabla q)_T - \langle\mathbf{R} \bm{n}\cdot\widehat{\bm{u}}_h,q\rangle_{\partial T}=0,\quad \forall q\in \mathcal{P}_k(T).
	$$
	It thus follows from the first two equations of $\bm{u}_h^*$ that
	$$
	(\varepsilon^{-1}w_h,q)_T-(\bm{u}_h^*,\mathbf{R}\nabla q)_T - \langle\mathbf{R} \bm{n}\cdot {\bm{u}}_h^*,q\rangle_{\partial T}=0,\quad \forall q\in \mathcal{P}_k(T),
	$$
	which, after integration by parts, yields the desired result.
	\qed
\end{proof}
\begin{remark}
	It is established that $\varepsilon\nabla\cdot(\mathbf{R}\bm{u}_h^*)$ exhibits identical accuracy and convergence rate to $w_h$, as demonstrated in \cite{nguyen2011hybridizable}.
\end{remark}

\begin{remark}
	The aforementioned postprocessing technique is applicable solely to 2D problems, as a straightforward replication of (\ref{eq:postprocess1}) cannot guarantee a well-defined $\bm{u}_h^*$ in 3D scenarios. Consequently, we present an alternative postprocessing approach that is valid for both 2D and 3D problems.
\end{remark}

\subsection{Alternative postprocessing scheme for both 2D and 3D problem}
Our focus lies on the following postprocessing scheme in 3D case, while noting that the 2D case can be handled similarly. 
It is curcial to note that $\dim\bm{\mathcal{P}}_{k+1}=\dim\nabla\times \bm{\mathcal{P}}_{k+1}+\dim \nabla \mathcal{P}_{k+2}$, as $\ker \nabla\times|_{{\bm{\mathcal{P}}_{k+1}}} = \nabla \mathcal{P}_{k+2}$. This leads to the following postprocessing scheme: Find $\bm{u}_h^\star\in \bm{\mathcal{P}}_{k+1}(T)$ for all $T\in \mathcal{T}_h$ such that
\begin{equation}\label{eq:post2}
	\begin{aligned}
	(\bm{u}_h^\star-\bm{u}_h,\nabla\times\bm{r})_T+\langle \bm{n}\times\bm{u}_h^\star  - \bm{n}\times\widehat{\bm{u}}_h,\bm{r}\rangle_{\partial T}&=0, \quad \forall \bm{r}\in \nabla\times\bm{\mathcal{P}}_{k+1}(T),\\
		(\bm{u}_h^\star-\bm{u}_h,\nabla v)_T&=0,\quad \forall v\in \mathcal{P}_{k+2}(T).
	\end{aligned}
\end{equation}
\begin{proposition}[well-definedness of (\ref{eq:post2})]
	The approximation $\bm{u}_h^\star$ is well-defined.
\end{proposition}
\begin{proof}
	By utilizing the relationship $\dim\bm{\mathcal{P}}_{k+1}=\dim\nabla\times \bm{\mathcal{P}}_{k+1}+\dim \nabla \mathcal{P}_{k+2}$, we can verify that the local problem \eqref{eq:post2} forms a square system. Hence, our objective is to demonstrate that $\bm{u}_h^\star=\bm{0}$ is the unique solution when $\widehat{\bm{u}}_h=\bm{0}$ and $\bm{u}_h=\bm{0}$. The first equation in \eqref{eq:post2} implies the following:
	$$
	0=(\bm{u}_h^\star,\nabla\times\bm{r})_T+\langle \bm{n}\times\bm{u}_h^\star ,\bm{r}\rangle_{\partial T}=(\nabla\times\bm{u}_h^\star,\bm{r}) \quad \forall \bm{r}\in \nabla\times \bm{\mathcal{P}}_{k+1}(T),
	$$
	which yields 
	$$
	\bm{u}_h^\star=\nabla\psi,\quad \text{ for some }\psi\in \mathcal{P}_{k+2}(T).
	$$
	By choosing $v=\psi$ in the second equation of \eqref{eq:post2}, we obtain $\nabla\psi=0$. This implies the desired result. \qed
\end{proof}

\begin{lemma}[postprocess property]\label{lem:posteq}
	The postprocessed solution $\bm{u}_h^\star$ satisfies
	$$
	(\nabla\times \bm{u}_h^\star,\bm{q})_T=(\varepsilon^{-1}\bm{w}_h,\bm{q})_T,\quad\forall \bm{q}\in \nabla\times \bm{\mathcal{P}}_{k+1}(T),\ T\in\mathcal{T}_h.
	$$
\end{lemma}
\begin{proof}
	By substituting the first equation of \eqref{eq:post2} into the HDG equation \eqref{eq:hdg1}, we can easily obtain the desired result.\qed
\end{proof}
\begin{remark}\label{rmk:postcurl}
	In the case of 3D problems, $\nabla\times \bm{\mathcal{P}}_{k+1}$ is a subspace of $\bm{\mathcal{P}}_k$, indicating that $\nabla\times\bm{u}_h^\star$ can only be the projection of $\varepsilon^{-1}\bm{w}h$ onto the largest divergence-free subspace of $\bm{\mathcal{P}}_k$. Conversely, for 2D problems, it is known that $\nabla\times \bm{\mathcal{P}}_{k+1} = \nabla\cdot (\mathbf{R}\bm{\mathcal{P}}_{k+1}) = \mathcal{P}_k$. As a result, we have $\nabla_h\times\bm{u}_h^\star = \varepsilon^{-1}\bm{w}_h$, which aligns with the result in Lemma \ref{lemma:post1}.
\end{remark}


Let $\bm{P}_h$ denote the $L^2$-projection onto $\nabla\times \bm{\mathcal{P}}_{k+1}$. Consequently, the projection is ${L}^2$-stable, and we can establish the following result.
\begin{lemma}[approximation property of $\bm{u}_h^\star$]
Suppose that  $\bm{w}=\nabla\times \bm{u}$ and $\bm{u}$ are the solution of \eqref{eq:mixed},  $\bm{u}_h$ is the solution of the discrete problem \eqref{eq:hdg} and $\bm{u}_h^\star$ is obtained by \eqref{eq:post2}. It holds that
	\begin{equation}\label{eq:approxpost}
	\begin{aligned}
	\varepsilon^{1/2}	\|\nabla\times(\bm{u}-\bm{u}_h^\star)\|_{L^2(T)}\le&  \inf_{\bm{{p}}_{k+1}\in\bm{\mathcal{P}}_{k+1}(T)}\varepsilon^{1/2}\|\nabla\times(\bm{u}-\bm{{p}}_{k+1})\|_{L^2(T)} \\
	&~+ \varepsilon^{-1/2}\|\bm{w}-\bm{w}_h\|_{L^2(T)} , \qquad \forall T\in \mathcal{T}_h.
		\end{aligned}
		\end{equation}
\end{lemma}
\begin{proof}
	Recalling that Lemma \ref{lem:posteq}  gives $\nabla_h\times \bm{u}_h^\star=\varepsilon^{-1} \bm{P}_h \bm{w}_h$ and leveraging the $L^2$ stability and orthogonality properties of ${{\bm{P}}}_h$, we can derive the following expression:
	$$
	\begin{aligned}
		\|\nabla\times &(\bm{u}-\bm{u}_h^\star)\|_{L^2(T)} \\
		&\le \|(I-{{\bm{P}}}_h) \nabla\times \bm{u}\|_{L^2(T)}+\|{{\bm{P}}}_h\nabla\times( \bm{u} - \bm{u}_h^\star)\|_{L^2(T)}\\
		&=\|(I-\bm{P}_h)\nabla\times(\bm{u}-\bm{{p}}_{k+1})\|_{L^2(T)}+\varepsilon^{-1}\|{{\bm{P}}}_h(\bm{w}-\bm{w}_h)\|_{L^2(T)}\\
		&\le \|\nabla\times(\bm{u}-\bm{{p}}_{k+1})\|_{L^2(T)}+\varepsilon^{-1}\|\bm{w}-\bm{w}_h\|_{L^2(T)},
	\end{aligned}
	$$
	for any $\bm{{p}}_{k+1}\in \bm{\mathcal{P}}_{k+1}(T)$. The desired result is obtained by multiplying both sides by $\varepsilon^{1/2}$.
	\qed
\end{proof}
\begin{remark}
	For 2D problems, the operator $\bm{P}_h$ can be replaced by the identity operator, resulting in the vanishing of the first term on the left side. This observation aligns with the fact that $\varepsilon\nabla\cdot(\mathbf{R}\bm{u}_h^\star)$ exhibits identical accuracy and convergence rate to $w_h$.
	\end{remark}
	\begin{remark}
	In the case of 3D problems, the first term on the right side of \eqref{eq:approxpost} yields
	$$
	\inf_{\bm{{p}}_{k+1}\in\bm{\mathcal{P}}_{k+1}(T)}\varepsilon^{1/2}\|\nabla\times(\bm{u}-\bm{{p}}_{k+1})\|_{L^2(T)}\le C\varepsilon^{1/2}h^{s_w+1}|\bm{u}|_{H_{s_w+2}(T)},
	$$
	where $s_w\in[0,k]$. This result coincides with the first term on the right side of the energy error estimate \eqref{eq:error}. However, it is important to note that as a component of the energy norm, the energy norm estimate \eqref{eq:error} does not provide a precise characterization of $\|\varepsilon^{-1/2}(\bm{w}-\bm{w}_h)\|_{L^2({\Omega})}$. In fact, the dependence of $\|\varepsilon^{-1/2}(\bm{w}-\bm{w}_h)\|_{L^2({\Omega})}$ on both $\varepsilon$ and $h$ is more intricate, and we observe that $\nabla_h\times\bm{u}_h^\star$ achieves one order better approximation of $\nabla\times\bm{u}$ compared to $\nabla_h\times\bm{u}_h$ when $\varepsilon \sim\mathcal{O}(1)$. Moreover, the accuracy is slightly enhanced with the same order when $\varepsilon\ll h$. These observations are numerically demonstrated in Section \ref{sec:numerical}.
\end{remark}

\section{Numerical results}\label{sec:numerical}
In this section, we present several numerical results for both three-dimensional problems \eqref{eq:mixed} and two-dimensional \eqref{eq:Hcurl-cd-2d} to validate our theoretical findings and display the performance for our proposed HDG method \eqref{eq:hdg} and postprocessing technique \eqref{eq:post2}. We use the uniform meshes with varying mesh sizes for the computational domain $\Omega$. For the stabilization parameters, we choose the following formulation that fulfills the requirements \eqref{eq:asstau1}-\eqref{eq:asstau4}:
$$
\begin{aligned}
	& \tau_t(F)=\max \left(\sup _{x \in F} \boldsymbol{\beta}(x) \cdot \boldsymbol{n}(x), 0\right)+\min \left(\frac{\varepsilon}{h_F}, 1\right), \quad \forall F \in \partial T, \forall T \in \mathcal{T}_h, \\
	& \tau_n(F)=\max\bigg(\max \left(\sup _{x \in F} \boldsymbol{\beta}(x) \cdot \boldsymbol{n}(x), 0\right),0.1\bigg), \quad \forall F \in \partial T, \forall T \in \mathcal{T}_h. \\
	&
\end{aligned}
$$
In order to study the convergence and accuracy of the methods we define the error in the broken $\bm{H}({\rm curl}; \mathcal{T}_h)$ semi-norm as
$$
\left|\boldsymbol{u}-\boldsymbol{u}_h\right|_{\boldsymbol{H}^c}:=\left(\sum_{T \in \mathcal{T}_h} \int_T\left\|\nabla \times\left(\boldsymbol{u}-\boldsymbol{u}_h\right)\right\|^2\right)^{1 / 2}.
$$
\subsection{Experiment I: 3D smooth solution}
We consider the case where $\boldsymbol{\beta}=(1,2,3)^T, \gamma=0$ ensuring the validity of  Assumptions \ref{ass:beta} and \ref{ass:reaction}   with $\rho(\bm{x}) \equiv 0$. The diffusion coefficient $\varepsilon$ is varied as $1,10^{-3}, 10^{-9}$.The forcing term $\boldsymbol{f}$ is chosen so that the analytical solution of \eqref{eq:original} is given by $\boldsymbol{u}(x, y, z)=$ $(\sin(y), \sin(z), \sin(x))^T$ in a unit cube $\Omega=(0,1)^3$ and the boundary data is determined accordingly based on the solution.

The convergence results are summarized in Table \ref{tab:energy3d} for the energy norm, showing a convergence rate of $k+\frac{1}{2}$ with respect to the mesh size $h$ for both diffusion-dominated and convection-dominated cases with $k=0,1$. These findings align with the result in Theorem \ref{thm:converge}. Additionally, Table \ref{tab:l23d} demonstrates $k+1$ order convergence of the $L^2$ norm for all $\varepsilon$ values when considering smooth solutions. Table \ref{tab:post3d} presents the convergence results for $\bm{u}_h$ and our proposed postprocessing solution $\bm{u}_h^\star$ in the $\bm{H}^c$ semi-norm. We observe that $\bm{u}_h^\star$ exhibits similar convergence behavior to $\bm{w}_h$, as anticipated by Lemma \ref{lem:posteq}. We observe that for the diffusion-dominated case ($\varepsilon=1$), $\bm{u}_h^\star$ achieves much better accuracy compared to $\bm{u}_h$, which demonstrates the effectiveness of our proposed postprocessing method. In the convection-dominated cases, the  $\bm{H}^c$ semi-norm convergence rate of $\bm{u}_h^\star$ matches that of $\bm{u}_h$, while still yielding slightly lower errors. In summary, the convergence of $\bm{u}_h^\star$ is dominated by the complex convergence behavior of $\bm{w}_h$, as evidenced by the results.

\begin{table}[!htbp]
	\centering
	
	\caption{Experiment I: History of convergence for $\triplenorm{(\bm{w}-\bm{w}_h,\bm{u}-\bm{u}_h,\bm{u}-\widehat{\bm{u}}_h)}$.}
	\label{tab:energy3d}
	\begin{tabular}{cccccccc}
		\hline
		\multirow{2}*{$k$} & \multirow{2}*{$1/h$}& \multicolumn{2}{c}{{$\varepsilon=1$}}&  \multicolumn{2}{c}{{$\varepsilon=10^{-3}$}}&  \multicolumn{2}{c}{{$\varepsilon=10^{-9}$}}\\
			\cline{3-8}
			\cline{3-8}
		& & Error & Order  & Error & Order  & Error & Order\\
		\hline
		\multirow{5}*{0}& 1	 & 1.56e+0	 & ---	 & 1.17e+0	 & --- 	 & 1.17e+0	 & --- \\
		& 2	 & 1.20e+0	 & 0.38	 & 9.95e-1	 & 0.23 	 & 9.95e-1	 & 0.23  \\
		& 4	 & 8.55e-1	 & 0.49	 & 7.47e-1	 & 0.41 	 & 7.48e-1	 & 0.41  \\
		& 8	 & 6.06e-1	 & 0.50	 & 5.43e-1	 & 0.46 	 & 5.45e-1	 & 0.46  \\
		& 16	 & 4.29e-1	 & 0.50	 & 3.88e-1	 & 0.48 	 & 3.92e-1	 & 0.48  \\
		\hline
		\multirow{5}*{1} & 1	 & 1.67e-1	 & ---	 & 1.18e-1	 & --- 	 & 1.18e-1	 & --- \\
		& 2	 & 5.96e-2	 & 1.49	 & 4.99e-2	 & 1.24 	 & 4.99e-2	 & 1.24  \\
		& 4	 & 2.07e-2	 & 1.53	 & 1.77e-2	 & 1.49 	 & 1.78e-2	 & 1.49  \\
		& 8	 & 7.26e-3	 & 1.51	 & 6.40e-3	 & 1.47 	 & 6.51e-3	 & 1.45  \\
		& 16	 & 2.55e-3	 & 1.51	 & 2.27e-3	 & 1.50 	 & 2.35e-3	 & 1.47  \\
		\hline
	\end{tabular}
\end{table}

\begin{table}[!htbp]
	\centering
  \caption{Experiment I: History of convergence for
  $\|\bm{u}-\bm{u}_h\|_{\mathcal{T}_h}$.}
	\label{tab:l23d}
	\begin{tabular}{cccccccc}
		\hline
		\multirow{2}*{$k$} & \multirow{2}*{$1/h$}& \multicolumn{2}{c}{{$\varepsilon=1$}}&  \multicolumn{2}{c}{{$\varepsilon=10^{-3}$}}&  \multicolumn{2}{c}{{$\varepsilon=10^{-9}$}}\\
		\cline{3-8}
		\cline{3-8}
		& & Error & Order  & Error & Order  & Error & Order\\
		\hline
		\multirow{5}*{0} & 1	 & 3.36e-1	 & ---	 & 3.89e-1	 & --- 	 & 3.89e-1	 & --- \\
		& 2	 & 1.66e-1	 & 1.02	 & 1.80e-1	 & 1.11 	 & 1.80e-1	 & 1.11  \\
		& 4	 & 8.45e-2	 & 0.97	 & 9.64e-2	 & 0.90 	 & 9.67e-2	 & 0.90  \\
		& 8	 & 4.28e-2	 & 0.98	 & 4.96e-2	 & 0.96 	 & 4.99e-2	 & 0.95  \\
		& 16	 & 2.15e-2	 & 0.99	 & 2.50e-2	 & 0.99 	 & 2.54e-2	 & 0.98  \\
		\hline
		\multirow{5}*{1} & 1	 & 2.51e-2	 & ---	 & 3.36e-2	 & --- 	 & 3.36e-2	 & --- \\
		& 2	 & 6.59e-3	 & 1.93	 & 9.08e-3	 & 1.89 	 & 9.12e-3	 & 1.88  \\
		& 4	 & 1.72e-3	 & 1.94	 & 2.35e-3	 & 1.95 	 & 2.38e-3	 & 1.94  \\
		& 8	 & 4.43e-4	 & 1.96	 & 6.11e-4	 & 1.95 	 & 6.27e-4	 & 1.93  \\
		& 16	 & 1.12e-4	 & 1.98	 & 1.53e-4	 & 1.99 	 & 1.62e-4	 & 1.95  \\
		\hline
	\end{tabular}
\end{table}

\begin{table}[!htbp]
	\centering
\caption{Experiment I: History convergence of the HDG methods and postprocessing
  method in the broken $\bm{H}^c$ semi-norm.}
	\label{tab:post3d}
	\begin{tabular}{cccccccc}
		\hline
		\multicolumn{8}{c}{$\varepsilon=1$}\\
		\hline
		\multirow{2}*{$k$} & \multirow{2}*{$1/h$}& \multicolumn{2}{c}{$\varepsilon^{-1}\|\bm{w}-\bm{w}_h\|_{\mathcal{T}_h}$}& \multicolumn{2}{c}{$|\bm{u}-\bm{u}_h|_{\bm{H}^c} $} & \multicolumn{2}{c}{$|\bm{u}-\bm{u}_h^\star|_{\bm{H}^c} $} \\
		\cline{3-8}
		& & Error & Order  & Error & Order  & Error & Order \\
		\hline
		\multirow{5}*{1} 
		& 1	 & 8.80e-2	 & --- 	 & 2.25e-2	 & ---	 & 8.80e-2	 & --- \\
		& 2	  & 2.50e-2	 & 1.81 	 & 9.89e-2	 & 0.89	 & 2.34e-2	 & 1.91  \\
		& 4	  & 7.83e-3	 & 1.68 	 & 5.02e-2	 & 0.98	 & 7.35e-3	 & 1.67  \\
		& 8	  & 2.51e-3	 & 1.64 	 & 2.52e-2	 & 0.99	 & 2.35e-3	 & 1.64  \\
		& 16	  & 8.28e-4	 & 1.60 	 & 1.26e-2	 & 1.00	 & 7.74e-4	 & 1.60  \\
		\hline
		\hline
		\multicolumn{8}{c}{$\varepsilon=10^{-3}$}\\
		\hline
		\multirow{2}*{$k$} & \multirow{2}*{$1/h$}& \multicolumn{2}{c}{$\varepsilon^{-1}\|\bm{w}-\bm{w}_h\|_{\mathcal{T}_h}$}& \multicolumn{2}{c}{$|\bm{u}-\bm{u}_h|_{\bm{H}^c} $} & \multicolumn{2}{c}{$|\bm{u}-\bm{u}_h^\star|_{\bm{H}^c} $} \\
		\cline{3-8}
		& & Error & Order  & Error & Order  & Error & Order \\
			\hline
		\multirow{5}*{1}& 1	 & 2.37e-1	 & --- 	 & 3.49e-2	 & ---	 & 2.34e-1	 & --- \\
		& 2	  & 1.05e-1	 & 1.18 	 & 1.05e-1	 & 0.89	 & 1.02e-1	 & 1.21  \\
		& 4	  & 4.94e-2	 & 1.09 	 & 5.37e-2	 & 0.97	 & 4.83e-2	 & 1.07  \\
		& 8	  & 2.38e-2	 & 1.05 	 & 2.72e-2	 & 0.98	 & 2.34e-2	 & 1.05  \\
		& 16	  & 1.10e-2	 & 1.11 	 & 1.37e-2	 & 0.99	 & 1.09e-2	 & 1.10  \\
		\hline
		\hline
		\multicolumn{8}{c}{$\varepsilon=10^{-9}$}\\
		\hline
		\multirow{2}*{$k$} & \multirow{2}*{$1/h$}& \multicolumn{2}{c}{$\varepsilon^{-1}\|\bm{w}-\bm{w}_h\|_{\mathcal{T}_h}$}& \multicolumn{2}{c}{$|\bm{u}-\bm{u}_h|_{\bm{H}^c} $} & \multicolumn{2}{c}{$|\bm{u}-\bm{u}_h^\star|_{\bm{H}^c} $} \\
		\cline{3-8}
		& & Error & Order  & Error & Order  & Error & Order \\
		\hline
		\multirow{5}*{1}& 1	 & 2.41e-1	 & --- 	 & 3.51e-2	 & ---	 & 2.39e-1	 & --- \\
		& 2	  & 1.08e-1	 & 1.16 	 & 1.05e-1	 & 0.89	 & 1.05e-1	 & 1.19  \\
		& 4	  & 5.23e-2	 & 1.05 	 & 5.38e-2	 & 0.97	 & 5.11e-2	 & 1.04  \\
		& 8	  & 2.65e-2	 & 0.98 	 & 2.73e-2	 & 0.98	 & 2.60e-2	 & 0.98  \\
		& 16	  & 1.34e-2	 & 0.98 	 & 1.38e-2	 & 0.99	 & 1.32e-2	 & 0.97  \\
		\hline
	\end{tabular}
\end{table}
\subsection{Experiment II: 2D smooth solution}
In this subsection, we consider the 2D problem \eqref{eq:Hcurl-cd-2d} with the parameters $\boldsymbol{\beta}=(1,2)^T$ and $\gamma=0$. We vary the diffusion coefficient $\varepsilon$ with values $1,10^{-3},10^{-9}$. The forcing term $\boldsymbol{f}$ is chosen such that the analytical solution of \eqref{eq:original} is given by $\boldsymbol{u}(x, y)=$ $(\sin(y),\sin(x))^T$ in the unit square $\Omega=(0,1)^2$, and the boundary data is determined accordingly.

In the 2D case, similar convergence results can be observed for both the energy norm and the $L^2$ norm, just as in the 3D case. Table \ref{tab:energy2d} demonstrates $k+\frac{1}{2}$ order convergence for the energy norm, while Table \ref{tab:l22d} shows $k+1$ order convergence for the $L^2$ norm, where $k=0,1,2$, across all $\varepsilon$ values. Additionally, we compare our proposed postprocessing solution, denoted as $\bm{u}_h^\star$, with the postprocessing solution defined in \eqref{eq:postprocess1} as $\bm{u}_h^*$.  Table \ref{tab:post2d} reveals that both $\bm{u}_h^*$ and $\bm{u}_h^\star$ exhibit the same convergence behavior in the broken $\bm{H}^c$ semi-norm, which aligns with the convergence of $\bm{w}_h$. This is consistent with the statement in Remark \ref{rmk:postcurl} that ${\rm curl}\bm{u}_h^\star$ and ${\rm curl}\bm{u}_h^*$ are equal to $\varepsilon^{-1}{w}_h$ for  2D problems. Both methods achieve much better accuracy in the diffusion-dominated case and slightly lower error in the convection-dominated case compared to $\bm{u}_h$. This highlights the effectiveness of the postprocessing methods, with the distinction that our method can also be applied to 3D problems. Finally, Table \ref{tab:post2dl2} presents the $L^2$ norm error of $\bm{u}_h^\star$ and $\bm{u}_h^*$ for $k=0$. We observe that our proposed postprocessing method achieves slightly better $L^2$ accuracy compared to $\bm{u}_h^*$.
\begin{table}[!htbp]
	\centering
  \caption{Experiment II: History of convergence for
  $\triplenorm{({w}-{w}_h,\bm{u}-\bm{u}_h,\bm{u}-\widehat{\bm{u}}_h)}$.}
	\label{tab:energy2d}
	\begin{tabular}{cccccccc}
		\hline
		\multirow{2}*{$k$} & \multirow{2}*{$1/h$}& \multicolumn{2}{c}{{$\varepsilon=1$}}&  \multicolumn{2}{c}{{$\varepsilon=10^{-3}$}}&  \multicolumn{2}{c}{{$\varepsilon=10^{-9}$}}\\
		\cline{3-8}
		\cline{3-8}
		& & Error & Order  & Error & Order  & Error & Order\\
		\hline
		\multirow{6}*{0}& 2	 & 7.77e-1	 & ---	 & 6.29e-1	 & --- 	 & 6.29e-1	 & --- \\
		& 4	 & 5.13e-1	 & 0.60	 & 4.10e-1	 & 0.62 	 & 4.11e-1	 & 0.61  \\
		& 8	 & 3.64e-1	 & 0.50	 & 2.95e-1	 & 0.47 	 & 2.97e-1	 & 0.47  \\
		& 16	 & 2.58e-1	 & 0.50	 & 2.10e-1	 & 0.49 	 & 2.12e-1	 & 0.48  \\
		& 32	 & 1.82e-1	 & 0.50	 & 1.49e-1	 & 0.50 	 & 1.51e-1	 & 0.49  \\
		& 64	 & 1.29e-1	 & 0.50	 & 1.05e-1	 & 0.50 	 & 1.08e-1	 & 0.49  \\
		\hline
		\multirow{6}*{1} & 2	 & 2.85e-2	 & ---	 & 2.76e-2	 & --- 	 & 2.76e-2	 & --- \\
		& 4	 & 1.15e-2	 & 1.31	 & 9.19e-3	 & 1.59 	 & 9.21e-3	 & 1.59  \\
		& 8	 & 4.06e-3	 & 1.50	 & 3.27e-3	 & 1.49 	 & 3.29e-3	 & 1.49  \\
		& 16	 & 1.43e-3	 & 1.50	 & 1.16e-3	 & 1.50 	 & 1.17e-3	 & 1.49  \\
		& 32	 & 5.06e-4	 & 1.50	 & 4.07e-4	 & 1.51 	 & 4.15e-4	 & 1.50  \\
		& 64	 & 1.79e-4	 & 1.50	 & 1.43e-4	 & 1.51 	 & 1.47e-4	 & 1.50  \\
		\hline
		\multirow{6}*{2}
		& 2	 & 1.55e-3	 & ---	 & 1.44e-3	 & --- 	 & 1.45e-3	 & --- \\
		& 4	 & 3.07e-4	 & 2.34	 & 2.50e-4	 & 2.53 	 & 2.55e-4	 & 2.50  \\
		& 8	 & 5.52e-5	 & 2.48	 & 4.45e-5	 & 2.49 	 & 4.62e-5	 & 2.47  \\
		& 16	 & 9.68e-6	 & 2.51	 & 7.71e-6	 & 2.53 	 & 8.22e-6	 & 2.49  \\
		& 32	 & 1.69e-6	 & 2.52	 & 1.32e-6	 & 2.55 	 & 1.45e-6	 & 2.50  \\
		& 64	 & 2.94e-7	 & 2.52	 & 2.24e-7	 & 2.56 	 & 2.57e-7	 & 2.50  \\
		\hline
	\end{tabular}
\end{table}

\begin{table}[!htbp]
	\centering
  \caption{Experiment II: History of convergence for
  $\|\bm{u}-\bm{u}_h\|_{\mathcal{T}_h}$.}
  \label{tab:l22d}
	\begin{tabular}{cccccccc}
		\hline
		\multirow{2}*{$k$} & \multirow{2}*{$1/h$}& \multicolumn{2}{c}{{$\varepsilon=1$}}&  \multicolumn{2}{c}{{$\varepsilon=10^{-3}$}}&  \multicolumn{2}{c}{{$\varepsilon=10^{-9}$}}\\
		\cline{3-8}
		\cline{3-8}
		& & Error & Order  & Error & Order  & Error & Order\\
		\hline
		\multirow{6}*{0} & 2	 & 1.72e-1	 & ---	 & 1.77e-1	 & --- 	 & 1.77e-1	 & --- \\
		& 4	 & 8.86e-2	 & 0.96	 & 9.58e-2	 & 0.89 	 & 9.63e-2	 & 0.88  \\
		& 8	 & 4.47e-2	 & 0.99	 & 4.88e-2	 & 0.97 	 & 4.93e-2	 & 0.97  \\
		& 16	 & 2.25e-2	 & 0.99	 & 2.44e-2	 & 1.00 	 & 2.49e-2	 & 0.98  \\
		& 32	 & 1.13e-2	 & 1.00	 & 1.21e-2	 & 1.02 	 & 1.26e-2	 & 0.99  \\
		& 64	 & 5.65e-3	 & 1.00	 & 5.89e-3	 & 1.04 	 & 6.32e-3	 & 0.99  \\
		\hline
		\multirow{6}*{1} & 2	 & 6.12e-3	 & ---	 & 7.80e-3	 & --- 	 & 7.81e-3	 & --- \\
		& 4	 & 1.83e-3	 & 1.74	 & 2.26e-3	 & 1.79 	 & 2.28e-3	 & 1.78  \\
		& 8	 & 4.65e-4	 & 1.97	 & 5.72e-4	 & 1.98 	 & 5.82e-4	 & 1.97  \\
		& 16	 & 1.17e-4	 & 1.99	 & 1.42e-4	 & 2.01 	 & 1.47e-4	 & 1.98  \\
		& 32	 & 2.95e-5	 & 1.99	 & 3.47e-5	 & 2.04 	 & 3.70e-5	 & 1.99  \\
		& 64	 & 7.38e-6	 & 2.00	 & 8.29e-6	 & 2.07 	 & 9.28e-6	 & 2.00  \\
		\hline
		\multirow{6}*{2} & 2	 & 4.51e-4	 & ---	 & 5.41e-4	 & --- 	 & 5.42e-4	 & --- \\
		& 4	 & 6.73e-5	 & 2.74	 & 7.12e-5	 & 2.93 	 & 7.18e-5	 & 2.92  \\
		& 8	 & 8.54e-6	 & 2.98	 & 8.88e-6	 & 3.00 	 & 9.01e-6	 & 2.99  \\
		& 16	 & 1.07e-6	 & 3.00	 & 1.10e-6	 & 3.01 	 & 1.13e-6	 & 3.00  \\
		& 32	 & 1.32e-7	 & 3.01	 & 1.36e-7	 & 3.02 	 & 1.41e-7	 & 3.00  \\
		& 64	 & 1.64e-8	 & 3.01	 & 1.67e-8	 & 3.02 	 & 1.76e-8	 & 3.00  \\
		\hline
	\end{tabular}
\end{table}

\begin{table}[!htbp]
	\centering
  \caption{Experiment II: History convergence of the HDG methods and postprocessing
  methods in the broken $\bm{H}^c$ semi-norm.}
	\label{tab:post2d}
	\begin{tabular}{cccccccccc}
		\hline
		\multicolumn{10}{c}{$\varepsilon=1$}\\
		\hline
		\multirow{2}*{$k$} & \multirow{2}*{$1/h$}& \multicolumn{2}{c}{$\varepsilon^{-1}\|{w}-{w}_h\|_{\mathcal{T}_h}$}& \multicolumn{2}{c}{$|\bm{u}-\bm{u}_h|_{\bm{H}^c} $}& \multicolumn{2}{c}{$|\bm{u}-\bm{u}_h^*|_{\bm{H}^c} $} & \multicolumn{2}{c}{$|\bm{u}-\bm{u}_h^\star|_{\bm{H}^c} $} \\
		\cline{3-10}
		& & Error & Order  & Error & Order  & Error & Order & Error & Order\\
		\hline
		\multirow{6}*{1} & 2	 & 1.23e-2	 & --- 	 & 1.04e-1	 & --- 	 & 1.23e-2	 & ---	 & 1.23e-2	 & --- \\
		& 4	 & 3.25e-3	 & 1.92 	 & 4.58e-2	 & 1.18 	 & 3.25e-3	 & 1.92	 & 3.25e-3	 & 1.92  \\
		& 8	 & 8.41e-4	 & 1.95 	 & 2.33e-2	 & 0.98 	 & 8.41e-4	 & 1.95	 & 8.41e-4	 & 1.95  \\
		& 16	 & 2.14e-4	 & 1.98 	 & 1.17e-2	 & 0.99 	 & 2.14e-4	 & 1.98	 & 2.14e-4	 & 1.98  \\
		& 32	 & 5.39e-5	 & 1.99 	 & 5.87e-3	 & 1.00 	 & 5.39e-5	 & 1.99	 & 5.39e-5	 & 1.99  \\
		& 64	 & 1.35e-5	 & 1.99 	 & 2.94e-3	 & 1.00 	 & 1.35e-5	 & 1.99	 & 1.35e-5	 & 1.99  \\
		\hline
		\multirow{6}*{2}& 2	 & 5.86e-4	 & --- 	 & 9.29e-3	 & --- 	 & 5.86e-4	 & ---	 & 5.86e-4	 & --- \\
		& 4	 & 6.52e-5	 & 3.17 	 & 2.14e-3	 & 2.12 	 & 6.52e-5	 & 3.17	 & 6.52e-5	 & 3.17  \\
		& 8	 & 8.07e-6	 & 3.01 	 & 5.34e-4	 & 2.00 	 & 8.07e-6	 & 3.01	 & 8.07e-6	 & 3.01  \\
		& 16	 & 1.00e-6	 & 3.01 	 & 1.33e-4	 & 2.00 	 & 1.00e-6	 & 3.01	 & 1.00e-6	 & 3.01  \\
		& 32	 & 1.25e-7	 & 3.00 	 & 3.33e-5	 & 2.00 	 & 1.25e-7	 & 3.00	 & 1.25e-7	 & 3.00  \\
		& 64	 & 1.56e-8	 & 3.00 	 & 8.33e-6	 & 2.00 	 & 1.56e-8	 & 3.00	 & 1.56e-8	 & 3.00  \\
		\hline
		\hline
		\multicolumn{10}{c}{$\varepsilon=10^{-3}$}\\
		\hline
		\multirow{2}*{$k$} & \multirow{2}*{$1/h$}& \multicolumn{2}{c}{$\varepsilon^{-1}\|\bm{w}-\bm{w}_h\|_{\mathcal{T}_h}$}& \multicolumn{2}{c}{$|\bm{u}-\bm{u}_h|_{\bm{H}^c} $}& \multicolumn{2}{c}{$|\bm{u}-\bm{u}_h^*|_{\bm{H}^c} $} & \multicolumn{2}{c}{$|\bm{u}-\bm{u}_h^\star|_{\bm{H}^c} $} \\
		\cline{3-10}
		& & Error & Order  & Error & Order  & Error & Order & Error & Order\\
		\hline
		\multirow{6}*{1} & 2	 & 3.48e-2	 & --- 	 & 1.03e-1	 & --- 	 & 3.48e-2	 & ---	 & 3.48e-2	 & --- \\
		& 4	 & 3.01e-2	 & 0.21 	 & 4.52e-2	 & 1.19 	 & 3.01e-2	 & 0.21	 & 3.01e-2	 & 0.21  \\
		& 8	 & 1.49e-2	 & 1.02 	 & 2.30e-2	 & 0.97 	 & 1.49e-2	 & 1.02	 & 1.49e-2	 & 1.02  \\
		& 16	 & 6.99e-3	 & 1.09 	 & 1.16e-2	 & 0.99 	 & 6.99e-3	 & 1.09	 & 6.99e-3	 & 1.09  \\
		& 32	 & 3.09e-3	 & 1.18 	 & 5.79e-3	 & 1.00 	 & 3.09e-3	 & 1.18	 & 3.09e-3	 & 1.18  \\
		& 64	 & 1.27e-3	 & 1.28 	 & 2.89e-3	 & 1.00 	 & 1.27e-3	 & 1.28	 & 1.27e-3	 & 1.28  \\
		\hline
		\multirow{6}*{2}& 2	 & 3.66e-3	 & --- 	 & 1.05e-2	 & --- 	 & 3.66e-3	 & ---	 & 3.66e-3	 & --- \\
		& 4	 & 1.39e-3	 & 1.40 	 & 2.42e-3	 & 2.12 	 & 1.39e-3	 & 1.40	 & 1.39e-3	 & 1.40  \\
		& 8	 & 3.23e-4	 & 2.10 	 & 6.08e-4	 & 1.99 	 & 3.23e-4	 & 2.10	 & 3.23e-4	 & 2.10  \\
		& 16	 & 6.94e-5	 & 2.22 	 & 1.52e-4	 & 2.00 	 & 6.94e-5	 & 2.22	 & 6.94e-5	 & 2.22  \\
		& 32	 & 1.37e-5	 & 2.34 	 & 3.80e-5	 & 2.00 	 & 1.37e-5	 & 2.34	 & 1.37e-5	 & 2.34  \\
		& 64	 & 2.54e-6	 & 2.43 	 & 9.44e-6	 & 2.01 	 & 2.54e-6	 & 2.43	 & 2.54e-6	 & 2.43  \\
		\hline
		\hline
		\multicolumn{10}{c}{$\varepsilon=10^{-9}$}\\
		\hline
		\multirow{2}*{$k$} & \multirow{2}*{$1/h$}& \multicolumn{2}{c}{$\varepsilon^{-1}\|\bm{w}-\bm{w}_h\|_{\mathcal{T}_h}$}& \multicolumn{2}{c}{$|\bm{u}-\bm{u}_h|_{\bm{H}^c} $}& \multicolumn{2}{c}{$|\bm{u}-\bm{u}_h^*|_{\bm{H}^c} $} & \multicolumn{2}{c}{$|\bm{u}-\bm{u}_h^\star|_{\bm{H}^c} $} \\
		\cline{3-10}
		& & Error & Order  & Error & Order  & Error & Order & Error & Order\\
		\hline
		\multirow{6}*{1} & 2	 & 3.52e-2	 & --- 	 & 1.03e-1	 & --- 	 & 3.52e-2	 & ---	 & 3.52e-2	 & --- \\
		& 4	 & 3.17e-2	 & 0.15 	 & 4.52e-2	 & 1.19 	 & 3.17e-2	 & 0.15	 & 3.17e-2	 & 0.15  \\
		& 8	 & 1.63e-2	 & 0.96 	 & 2.30e-2	 & 0.97 	 & 1.63e-2	 & 0.96	 & 1.63e-2	 & 0.96  \\
		& 16	 & 8.27e-3	 & 0.98 	 & 1.16e-2	 & 0.99 	 & 8.27e-3	 & 0.98	 & 8.27e-3	 & 0.98  \\
		& 32	 & 4.17e-3	 & 0.99 	 & 5.83e-3	 & 0.99 	 & 4.17e-3	 & 0.99	 & 4.17e-3	 & 0.99  \\
		& 64	 & 2.09e-3	 & 0.99 	 & 2.92e-3	 & 1.00 	 & 2.09e-3	 & 0.99	 & 2.09e-3	 & 0.99  \\
		\hline
		\multirow{6}*{2}& 2	 & 3.72e-3	 & --- 	 & 1.06e-2	 & --- 	 & 3.72e-3	 & ---	 & 3.72e-3	 & --- \\
		& 4	 & 1.55e-3	 & 1.27 	 & 2.43e-3	 & 2.12 	 & 1.55e-3	 & 1.27	 & 1.55e-3	 & 1.27  \\
		& 8	 & 3.99e-4	 & 1.95 	 & 6.09e-4	 & 1.99 	 & 3.99e-4	 & 1.95	 & 3.99e-4	 & 1.95  \\
		& 16	 & 1.01e-4	 & 1.98 	 & 1.52e-4	 & 2.00 	 & 1.01e-4	 & 1.98	 & 1.01e-4	 & 1.98  \\
		& 32	 & 2.55e-5	 & 1.99 	 & 3.81e-5	 & 2.00 	 & 2.55e-5	 & 1.99	 & 2.55e-5	 & 1.99  \\
		& 64	 & 6.40e-6	 & 1.99 	 & 9.53e-6	 & 2.00 	 & 6.40e-6	 & 1.99	 & 6.40e-6	 & 1.99  \\
		\hline
	\end{tabular}
\end{table}

\begin{table}[!htbp]
	\centering
	\caption{Experiment II: History convergence of the postprocessing methods in the $L^2$ norm.}
	\label{tab:post2dl2}
	\begin{tabular}{cccccc}
		\hline
		\multicolumn{6}{c}{$\varepsilon=1$}\\
		\hline
		\multirow{2}*{$k$} & \multirow{2}*{$1/h$}& \multicolumn{2}{c}{$\|\bm{u}-\bm{u}_h^*\|_{\mathcal{T}_h}$}& \multicolumn{2}{c}{$\|\bm{u}-\bm{u}_h^\star\|_{\mathcal{T}_h}$} \\
		\cline{3-6}
		& & Error & Order  & Error & Order \\
		\hline
		\multirow{6}*{0} & 2	 & 2.48e-1	 & --- 	 & 1.72e-1	 & ---  \\
		& 4	 & 1.25e-1	 & 0.99 	 & 8.84e-2	 & 0.96   \\
		& 8	 & 6.27e-2	 & 0.99 	 & 4.46e-2	 & 0.99   \\
		& 16	 & 3.14e-2	 & 1.00 	 & 2.24e-2	 & 0.99   \\
		& 32	 & 1.57e-2	 & 1.00 	 & 1.12e-2	 & 1.00   \\
		& 64	 & 7.87e-3	 & 1.00 	 & 5.63e-3	 & 1.00   \\
		\hline
		\hline
		\multicolumn{6}{c}{$\varepsilon=10^{-3}$}\\
		\hline
		\multirow{2}*{$k$} & \multirow{2}*{$1/h$}& \multicolumn{2}{c}{$\|\bm{u}-\bm{u}_h^*\|_{\mathcal{T}_h}$}& \multicolumn{2}{c}{$\|\bm{u}-\bm{u}_h^\star\|_{\mathcal{T}_h}$} \\
		\cline{3-6}
		& & Error & Order  & Error & Order \\
		\hline
		\multirow{6}*{0} & 2	 & 2.63e-1	 & --- 	 & 1.79e-1	 & ---  \\
		& 4	 & 1.34e-1	 & 0.98 	 & 9.79e-2	 & 0.87   \\
		& 8	 & 6.80e-2	 & 0.97 	 & 4.97e-2	 & 0.98   \\
		& 16	 & 3.41e-2	 & 0.99 	 & 2.49e-2	 & 1.00   \\
		& 32	 & 1.69e-2	 & 1.01 	 & 1.22e-2	 & 1.02   \\
		& 64	 & 8.35e-3	 & 1.02 	 & 5.95e-3	 & 1.04   \\
		\hline
		\hline
		\multicolumn{6}{c}{$\varepsilon=10^{-9}$}\\
		\hline
		\multirow{2}*{$k$} & \multirow{2}*{$1/h$}& \multicolumn{2}{c}{$\|\bm{u}-\bm{u}_h^*\|_{\mathcal{T}_h}$}& \multicolumn{2}{c}{$\|\bm{u}-\bm{u}_h^\star\|_{\mathcal{T}_h}$} \\
		\cline{3-6}
		& & Error & Order  & Error & Order \\
		\hline
		\multirow{6}*{0} & 2	 & 2.63e-1	 & --- 	 & 1.79e-1	 & ---  \\
		& 4	 & 1.34e-1	 & 0.97 	 & 9.85e-2	 & 0.86   \\
		& 8	 & 6.86e-2	 & 0.97 	 & 5.04e-2	 & 0.97   \\
		& 16	 & 3.47e-2	 & 0.98 	 & 2.55e-2	 & 0.98   \\
		& 32	 & 1.75e-2	 & 0.99 	 & 1.29e-2	 & 0.99   \\
		& 64	 & 8.78e-3	 & 0.99 	 & 6.46e-3	 & 0.99   \\
		\hline
	\end{tabular}
\end{table}

\subsection{Experiment III: Rotating flow}
In this experiment, we consider the rotating flow problem with $\varepsilon=10^{-9}$, $\gamma=0$, $\bm{\beta}=(y-1/2,1/2-x)^T$ and the domain $\Omega=(0,1)^2$. The solution is prescribed along the slit $1/2\times[0,1/2]$ as follows:
$$
\bm{u}(1/2,y)=(	\sin^2(2\pi y),\sin^2(2\pi y))^T, \quad y\in [0,1/2].
$$
We present the results of the HDG method using different polynomial degrees $k=0,1,2$. Figure \ref{fig:rot} displays the first component of the HDG solution $\bm{u}_{h}$ on a uniform mesh with $h=1/16$. The HDG method demonstrates a good performance in this problem and it is evident that higher-order methods yield improved approximation results.
\begin{figure}[!htbp]
	\centering
	{\includegraphics[width=.45\textwidth,height=.23\textheight]{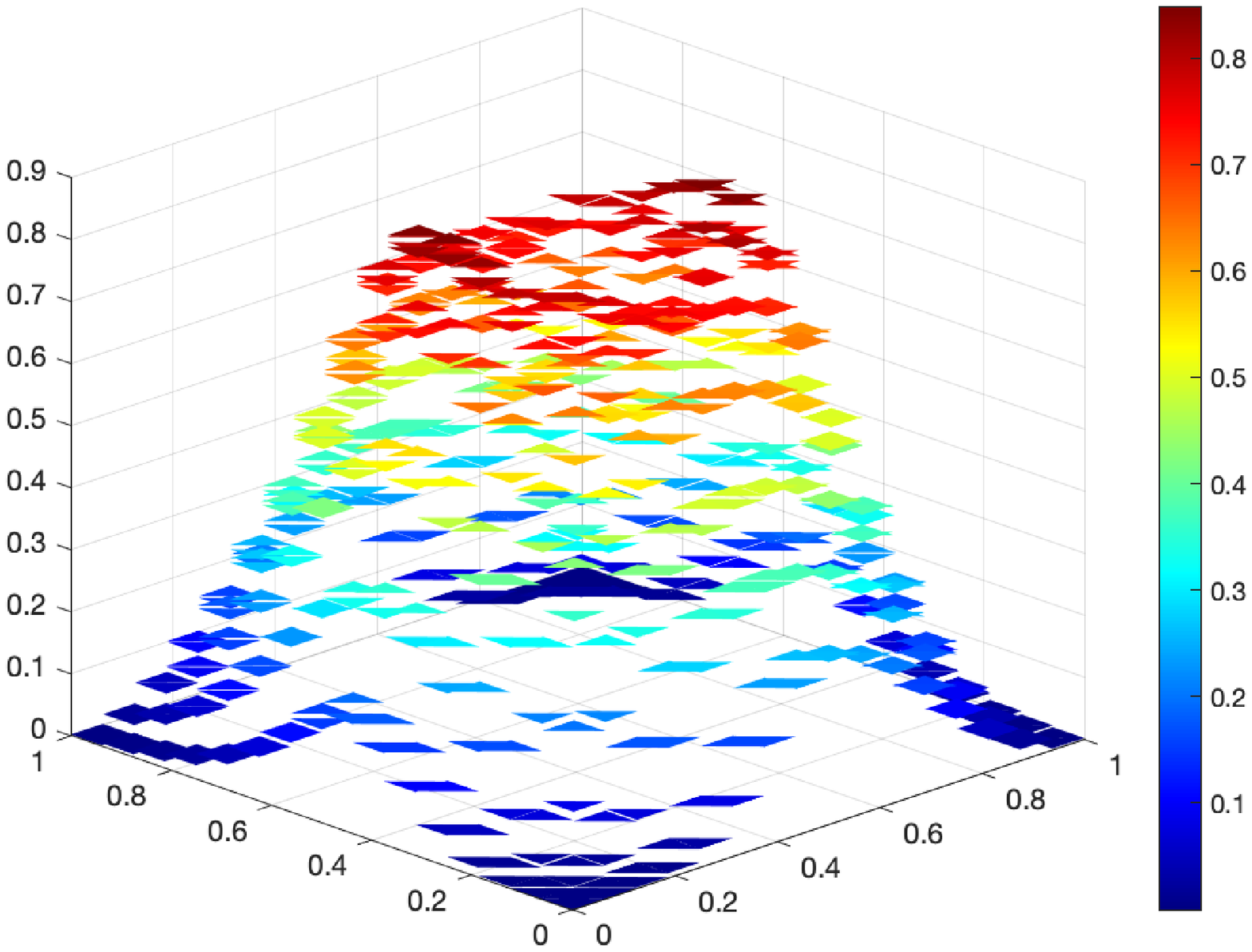}}
	{\includegraphics[width=.42\textwidth,height=.23\textheight]{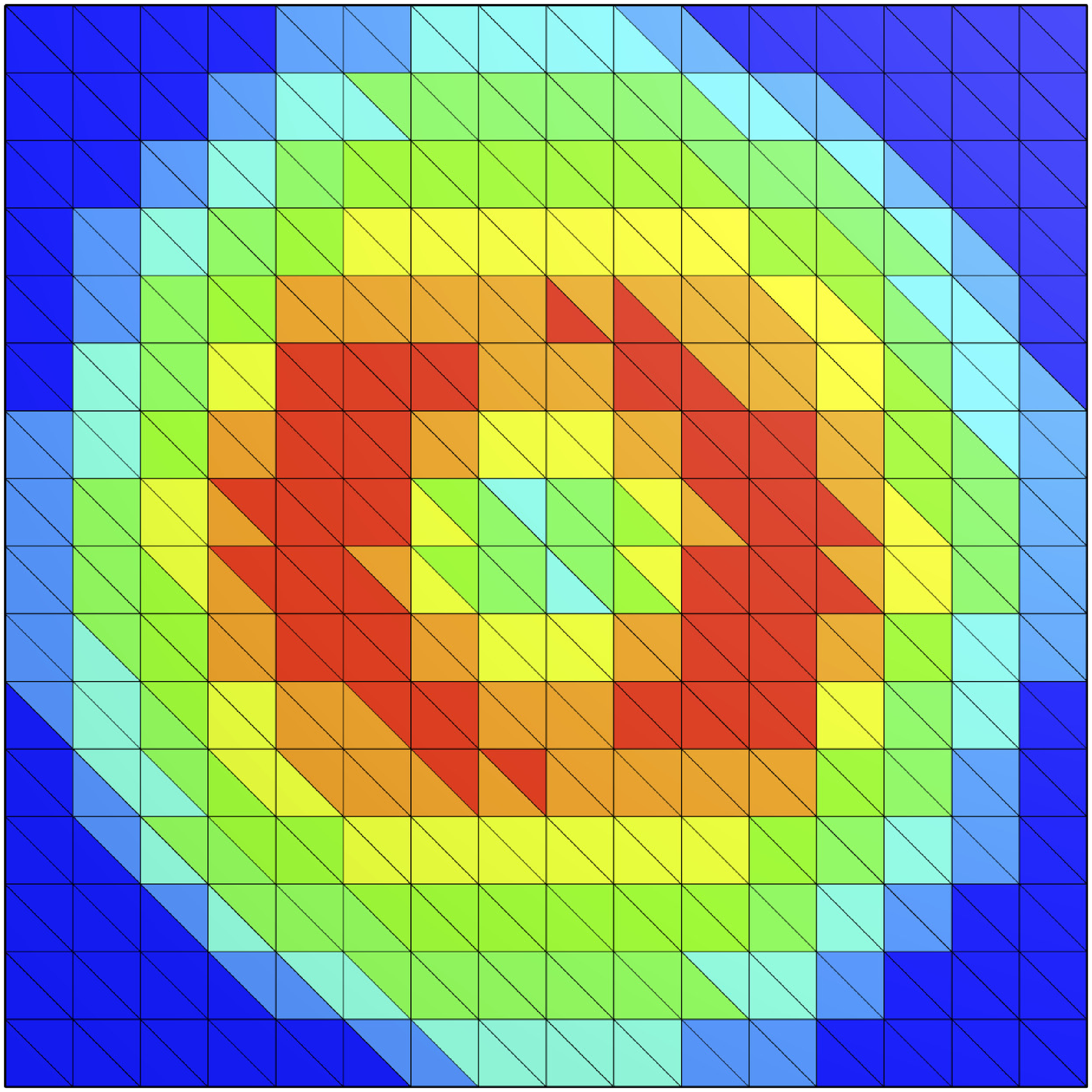}}
	{\includegraphics[width=.45\textwidth,height=.23\textheight]{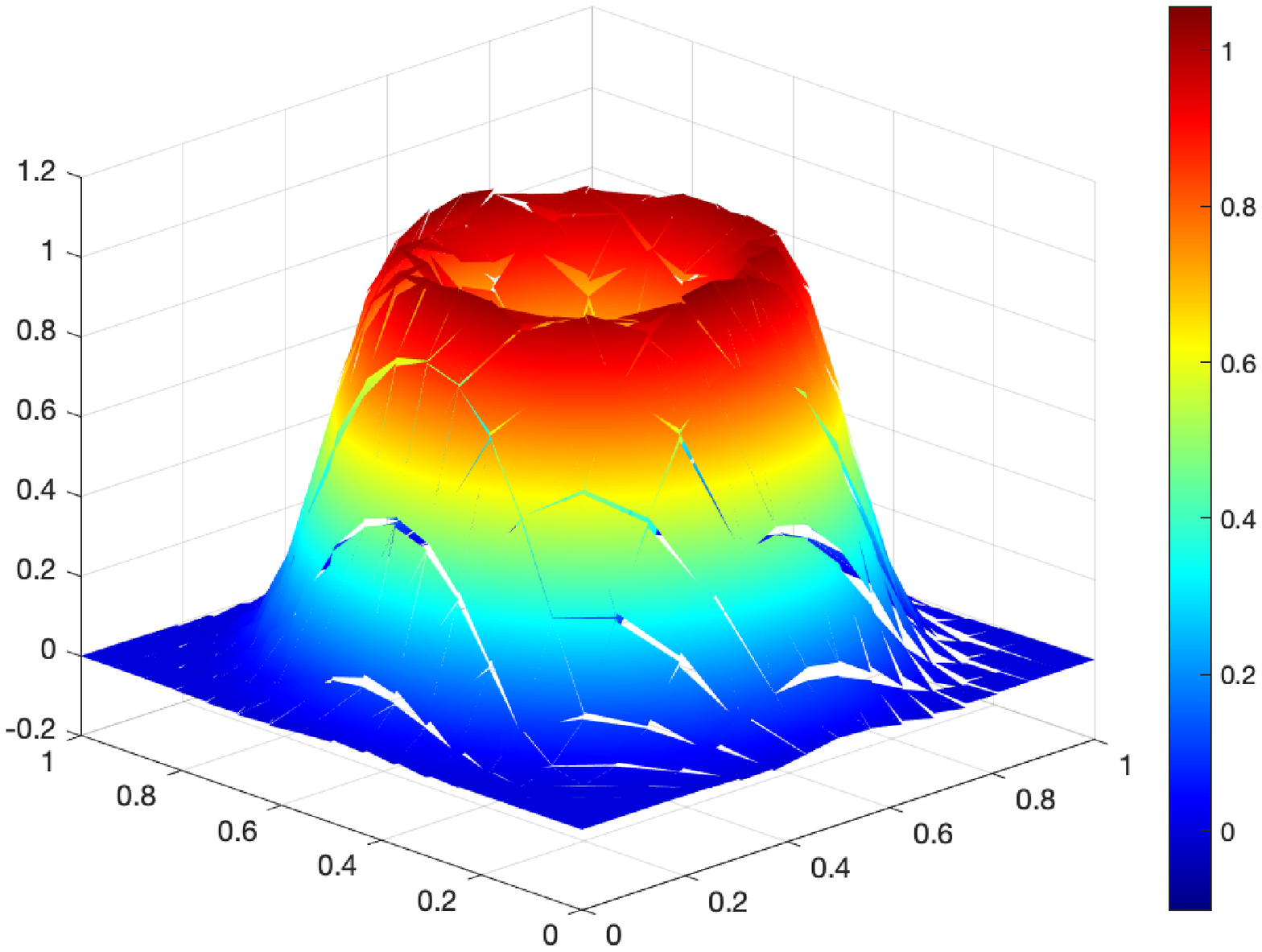}}
	{\includegraphics[width=.42\textwidth,height=.23\textheight]{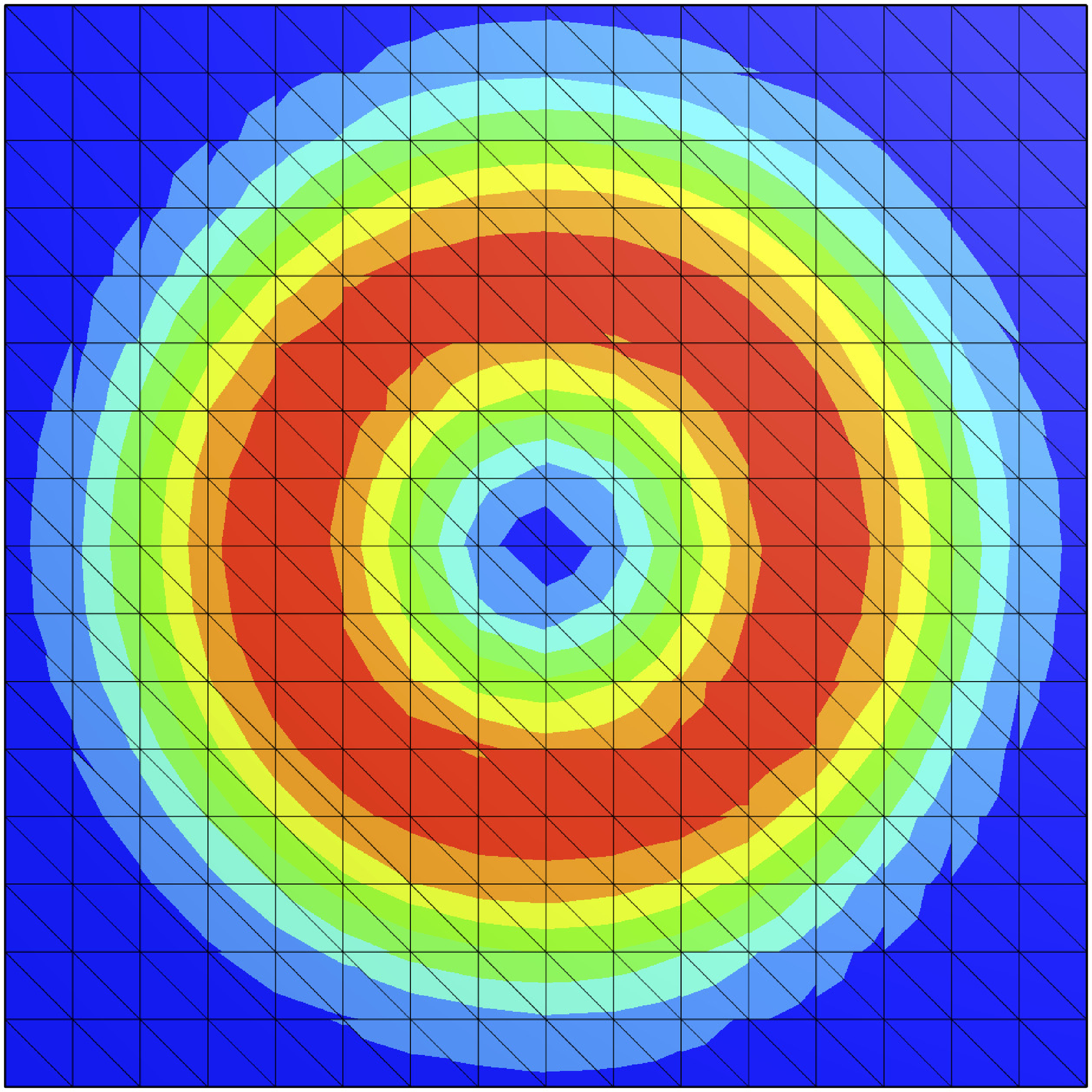}}
	{\includegraphics[width=.45\textwidth,height=.23\textheight]{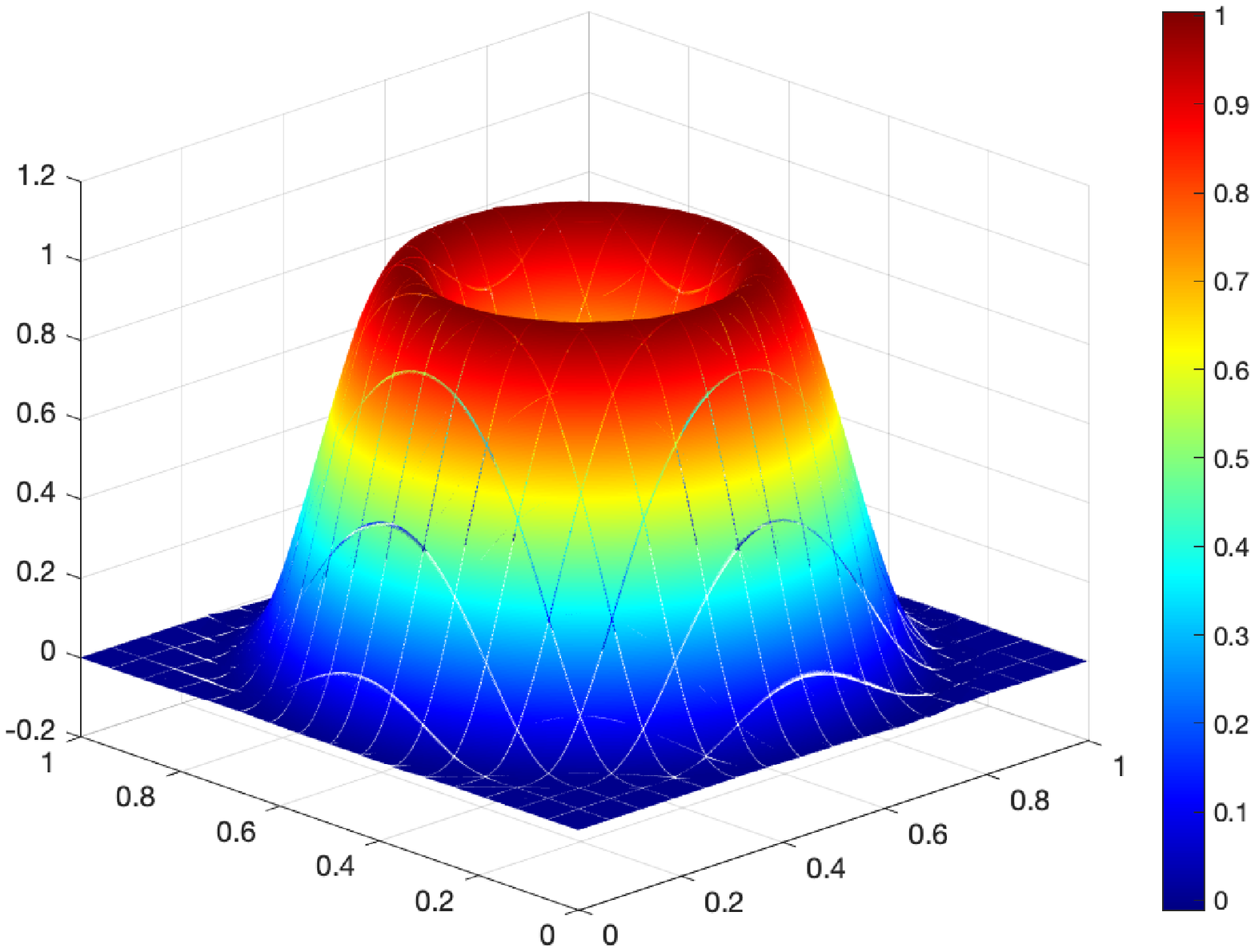}}
	{\includegraphics[width=.42\textwidth,height=.23\textheight]{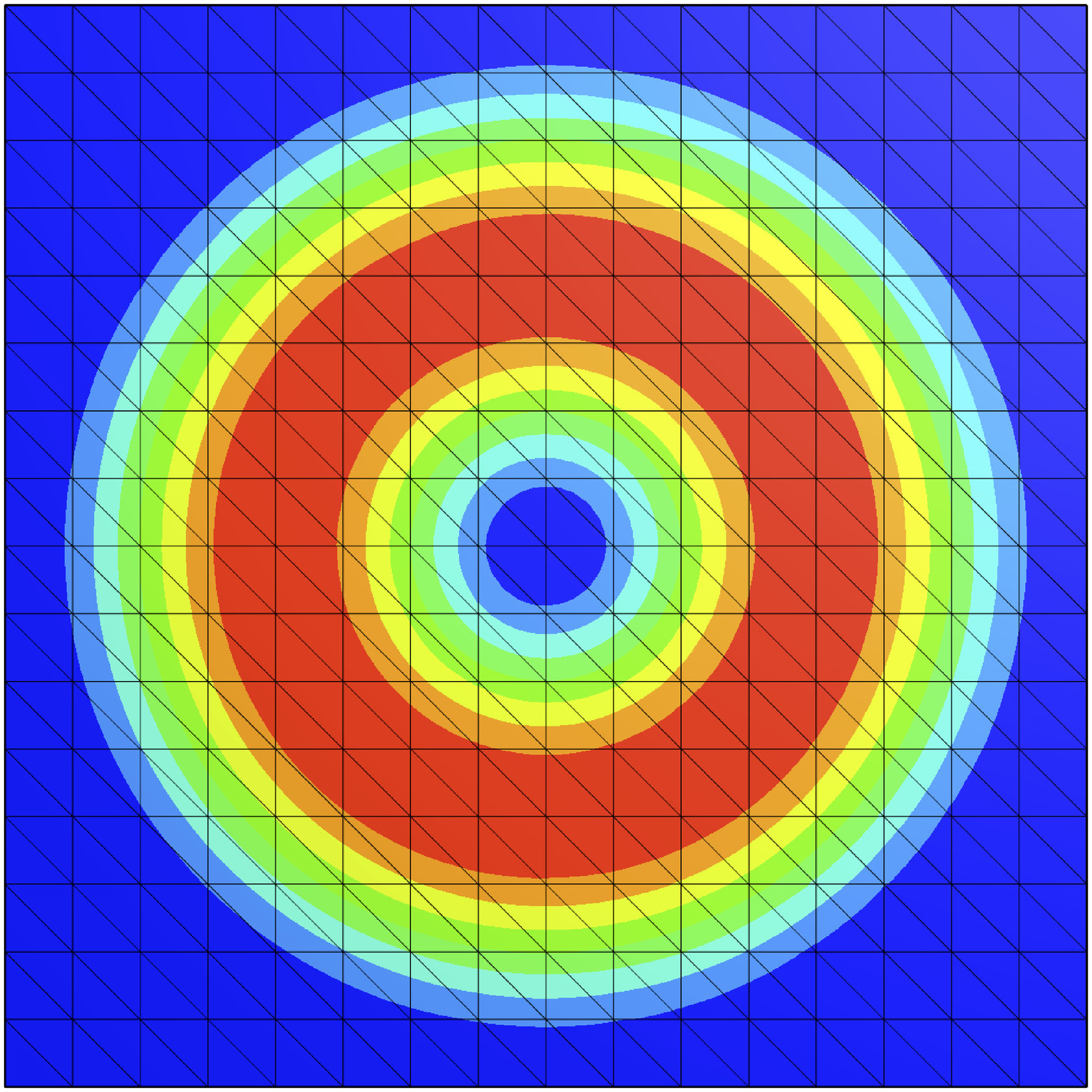}}
	\caption{Experiment III: A comparison of $\bm{\mathcal{P}}_0$-HDG (top), $\bm{\mathcal{P}}_1$-HDG (middle) and $\bm{\mathcal{P}}_2$-HDG (bottom) for a rotation flow.}
	\label{fig:rot}
\end{figure}

\subsection{Experiment IV: Interior layer problem}
The following example focuses on evaluating the performance of the HDG method when dealing with interior layers in 2D case. We consider the parameters $\gamma=0$, $\bm{\beta}=[1/2,\sqrt{3}/2]^T$, $\bm{f}=(0,0)^T$ on the domain $(0,1)^2$ with different diffusion coefficients $\varepsilon=10^{-3}$ and $10^{-9}$. The Dirichlet boundary conditions are imposed for the following function:
$$
\bm{u}=\left\{
\begin{aligned}
	&(1,1)^T, \text{ on } \{y=0,0\le x\le 1\},\\
	&(1,1)^T, \text{ on } \{x=0,0\le y\le 0.2\},\\
	&(0,0)^T,\text{ elsewhere. }
\end{aligned}
\right.
$$
In Figure \ref{fig:intlayer}, we plot the first component of the approximation solution $\bm{u}_h$. It is observed that the solution captures the presence of the interior layer but it also exhibits oscillations within the layer for all values of $\varepsilon$.
\begin{figure}[!htbp]
	\centering
	{\includegraphics[width=.45\textwidth]{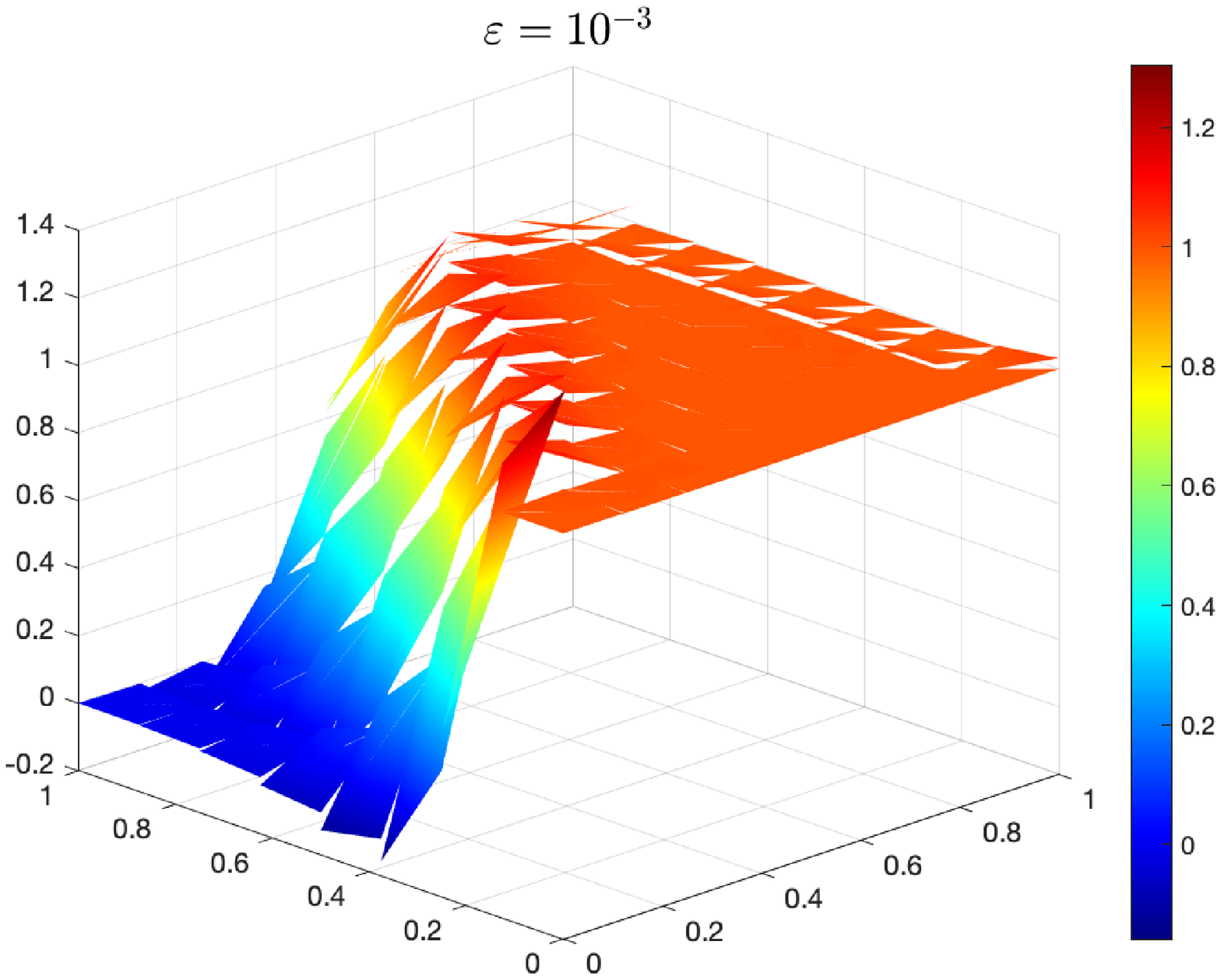}}
	{\includegraphics[width=.45\textwidth]{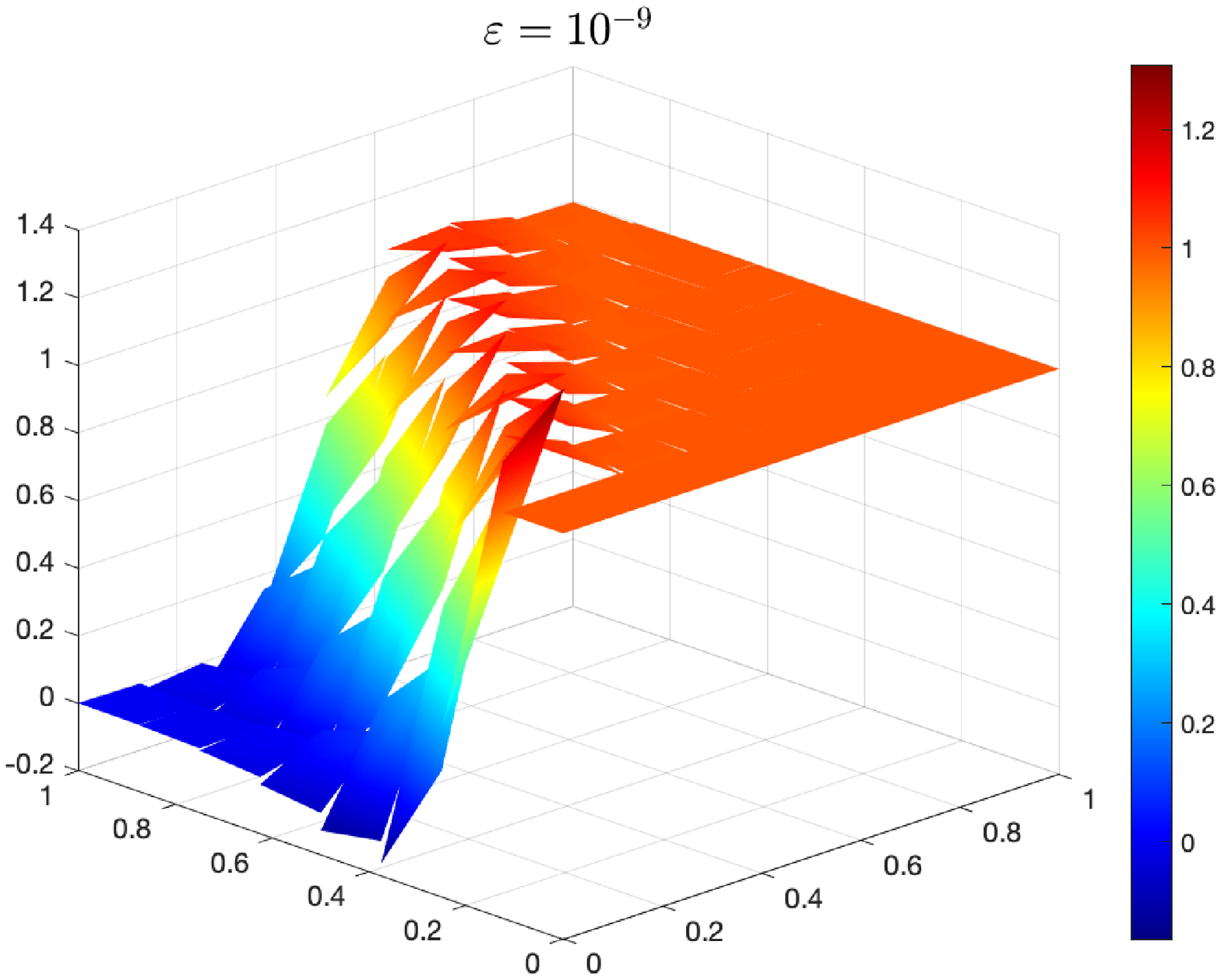}}
	\caption{Experiment IV: The first component of the HDG solution $\bm{u}_{h}$ for a interior layer problem with $\varepsilon=10^{-3}$ (left) and $\varepsilon=10^{-9}$ (right).}\label{fig:intlayer}
\end{figure}
\subsection{Experiment V: Boundary layer problem}
In this example, we utilize the HDG method to solve a 2D boundary layer problem. The problem is defined with the parameters $\gamma=0$, $\bm{\beta}=(1,2)^T$, and the domain $\Omega=(0,1)^2$ and we vary the diffusion coefficient $\varepsilon$ as $\varepsilon=10^{-3}$ and $10^{-9}$. The forcing term $\bm{f}$ is set as $(1,1)^T$, and homogeneous Dirichlet boundary conditions are imposed.

As shown in Figure \ref{fig:bdylayer}, the numerical solutions exhibit stability without spurious oscillations as $\varepsilon$ approaches zero.  Since the boundary conditions are imposed in a weak sense, the boundary layer is not resolved by the HDG approximations in this coarse mesh, which is also observed in the scalar case \cite{fu2015analysis}.
\begin{figure}[!htbp]
	\centering
	{\includegraphics[width=.45\textwidth]{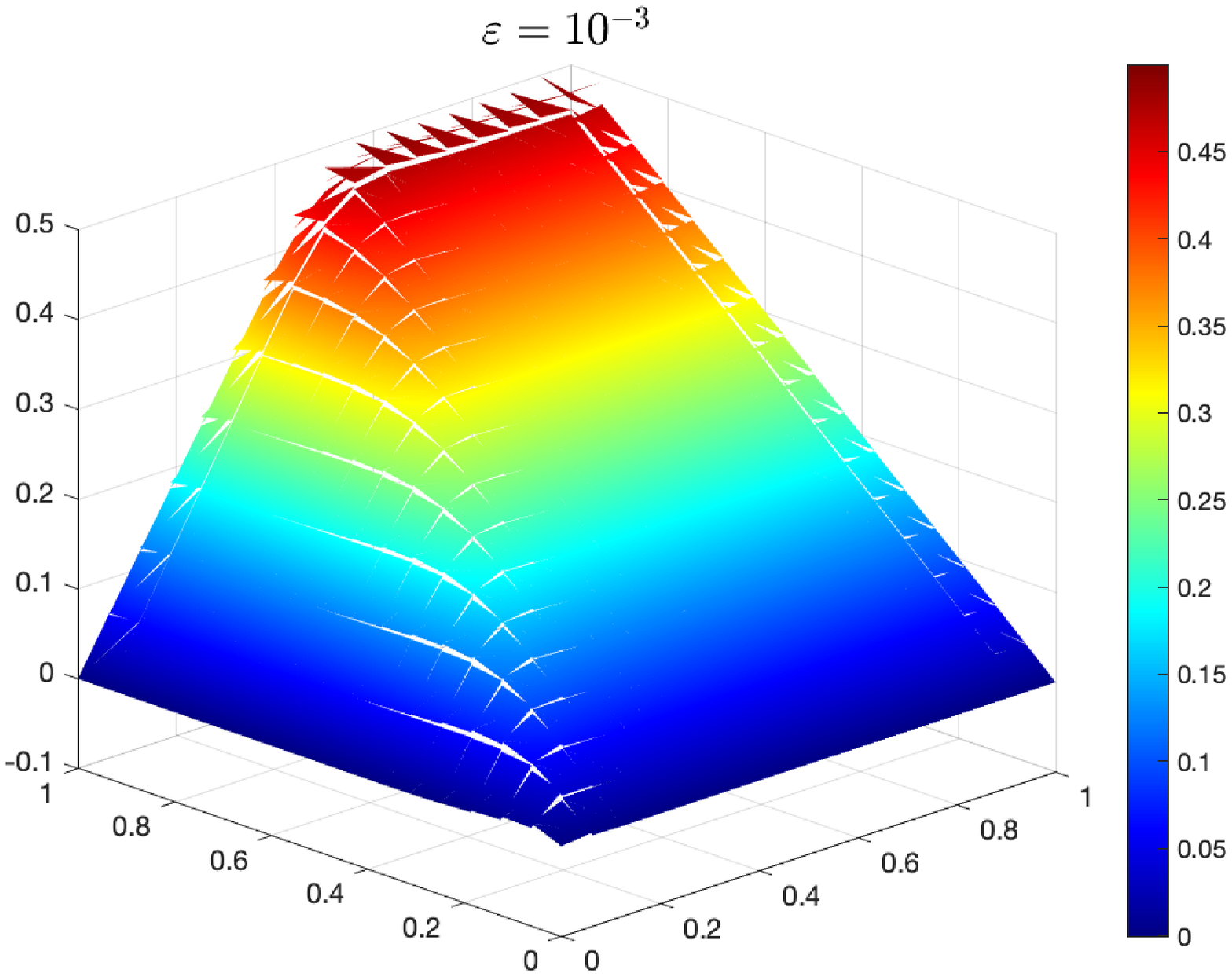}}
	{\includegraphics[width=.45\textwidth]{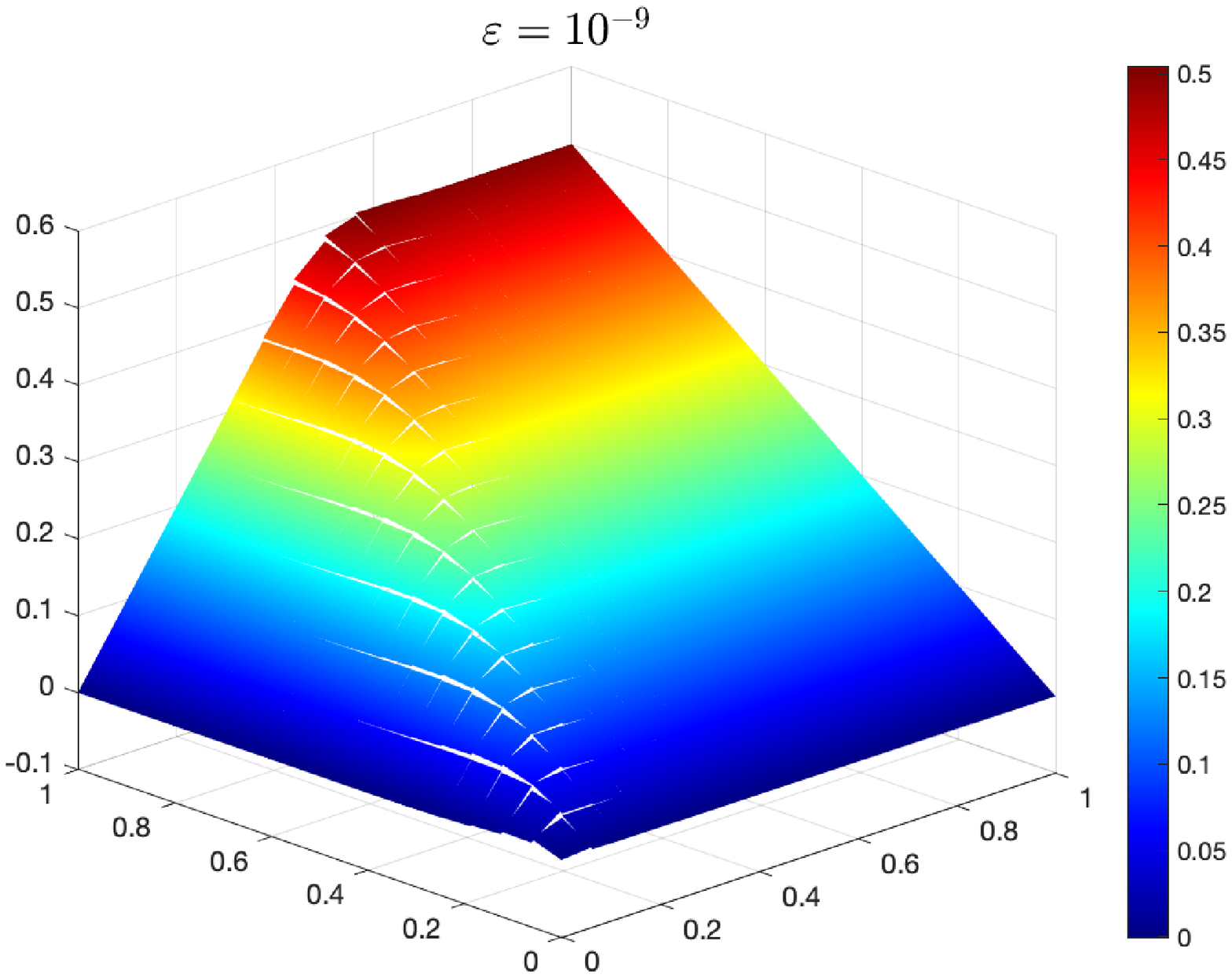}}
	\caption{Experiment V: The first component of the HDG solution $\bm{u}_{h}$ for a boundary layer problem with $\varepsilon=10^{-3}$ (left) and  $\varepsilon=10^{-9}$ (right).}\label{fig:bdylayer}
\end{figure}

\bibliographystyle{spmpsci}      
\bibliography{curl_HDG} 

\end{document}